\documentclass[reqno,12pt,oneside]{amsart}

\usepackage{amsthm,amsmath,amssymb,amsfonts}
\usepackage[colorlinks=true,hypertexnames=false,linkcolor=blue,citecolor=blue]{hyperref}
\usepackage[top=3cm,bottom=2.5cm,left=2.5cm, right=3cm,marginparwidth=0cm]{geometry}
\usepackage{epsfig,graphics}
\usepackage{amscd}
\usepackage{amsmath}
\usepackage{mathtools}
\usepackage{bm}
\usepackage{amssymb}
\usepackage{graphicx}
\usepackage{psfrag}
\usepackage{verbatim}
\usepackage{tikz}
\usetikzlibrary{quotes}
\usetikzlibrary{topaths}
\usetikzlibrary{shapes,arrows,cd}
\usepackage{adjustbox}

\newtheorem{theorem}{Theorem}[section]
\newtheorem{lemma}[theorem]{Lemma}
\newtheorem{prop}[theorem]{Proposition}
\newtheorem{coro}[theorem]{Corollary}

\theoremstyle{definition}
\newtheorem{definition}[theorem]{Definition}

\newtheorem{question}[theorem]{Question}

\theoremstyle{remark}
\newtheorem{remark}[theorem]{Remark}

\theoremstyle{plain}
\newtheorem{innercustomthm}{Corollary\!\!}
\newenvironment{customthm}[1]
  {\renewcommand\theinnercustomthm{#1}\innercustomthm}
  {\endinnercustomthm}

\numberwithin{equation}{section}

\theoremstyle{plain}
\newtheorem{maintheorem}{Theorem}

\newcommand{\Q}{\ensuremath{\mathbb{Q}}}
\newcommand{\R}{\ensuremath{\mathbb{R}}}
\newcommand{\Z}{\ensuremath{\mathbb{Z}}}

\newcommand{\nt}{\ensuremath{\mathbb{N}}}

\newcommand{\bt}{\operatorname{BT}}
\newcommand{\crd}{\operatorname{CrD}}

\newcommand{\varia}{\operatorname{var}}

\begin{document}

\title[Quasisymmetric orbit-flexibility]{Quasisymmetric orbit-flexibility \\of multicritical circle maps}

\author{Edson de Faria}
\address{Instituto de Matem\'atica e Estat\'istica, Universidade de S\~ao Paulo}
\curraddr{Rua do Mat\~ao 1010, 05508-090, S\~ao Paulo SP, Brasil}
\email{edson@ime.usp.br}

\author{Pablo Guarino}
\address{Instituto de Matem\'atica e Estat\'istica, Universidade Federal Fluminense}
\curraddr{Rua Prof. Marcos Waldemar de Freitas Reis, S/N, 24.210-201, Bloco H, Campus do Gragoat\'a, Niter\'oi, Rio de Janeiro RJ, Brasil}
\email{pablo\_\,guarino@id.uff.br}

\thanks{The first author has been supported by ``Projeto Tem\'atico Din\^amica em Baixas Dimens\~oes'' FAPESP Grant  2016/25053-8, while the second author has been supported by Conselho Nacional de Desenvolvimento Cient\'ifico e Tecnol\'ogico (CNPq) and by Coordena\c{c}\~ao de Aperfei\c{c}oamento de Pessoal de N\'ivel Superior - Brasil (CAPES) grant 23038.009189/2013-05.}



\begin{abstract} Two given orbits of a minimal circle homeomorphism $f$ are said to be \emph{geometrically equivalent} if there exists a \emph{quasisymmetric} circle homeomorphism identifying both orbits and commuting with $f$. By a well-known theorem due to Herman and Yoccoz, if $f$ is a smooth diffeomorphism with \emph{Diophantine} rotation number, then any two orbits are geometrically equivalent. As it follows from the \emph{a-priori bounds} of Herman and \'Swi\c{a}tek, the same holds if $f$ is a critical circle map with rotation number of \emph{bounded type}. By contrast, we prove in the present paper that if $f$ is a critical circle map whose rotation number belongs to a certain full Lebesgue measure set in $(0,1)$, then the number of equivalence classes is \emph{uncountable} (Theorem \ref{ThmA}). The proof of this result relies on the ergodicity of a two-dimensional skew product over the Gauss map. As a by-product of our techniques, we construct topological conjugacies between multicritical circle maps which are \emph{not} quasisymmetric, and we show that this phenomenon is \emph{abundant}, both from the topological and measure-theoretical viewpoints (Theorems \ref{ThmB} and \ref{ThmC}).
\end{abstract}

\maketitle

\vspace{-0.5cm}

\section{Introduction}

The dynamics of a minimal circle homeomorphism $f:S^1 \to S^1$ is topologically very homogeneous, in the sense that any two of its orbits look topologically the same. But are such orbits geometrically the same? This question is only meaningful if one properly defines the underlying concept of geometric equivalence. One also needs to assume that $f$ is sufficiently regular (i.e., has some reasonable degree of smoothness). Let us agree that the orbits $\mathcal{O}_f(x)$ and $\mathcal{O}_f(y)$ of two points $x,y \in S^1$ are \emph{geometrically equivalent} if there exists a self-conjugacy $h:S^1 \to S^1$ ($h \circ f=f \circ h$) which is a quasisymmetric homeomorphism carrying $\mathcal{O}_f(x)$ to $\mathcal{O}_f(y)$. So let us ask that question again: are two given orbits $\mathcal{O}_f(x)$ and $\mathcal{O}_f(y)$ geometrically equivalent?

The answer is easily seen to be ``yes'' if $f$ is smoothly conjugate to a rotation: this is the case, for instance, when $f$ is a smooth diffeomorphism with Diophantine rotation number (as it follows from the famous rigidity result of Herman \cite{hermanIHES} improved by Yoccoz \cite{yoccoz3}, and also by Katznelson and Ornstein \cite{KO}). In a sense to be made precise below, our main goal in the present paper is to show that the answer is ``almost always no'' when $f$ is a critical circle map. Precise statements will be given in sections \ref{secenunciadoThmA} to \ref{secboundedgeo} below.

Since the study presented here involves the notions of \emph{rigidity} and \emph{flexibility}, we proceed to say a few words about these concepts.

\subsubsection*{Rigidity} In one-dimensional dynamics, a current topic of research is to understand the connection, if any, between rigidity and renormalizability properties of interval or circle maps. For maps having a single critical point -- unimodal interval maps or critical circle homeomorphisms -- major advances
have been achieved in recent years, and a reasonably complete picture has emerged.
However, for maps having two or more critical points, much remains to be done.

In the present paper, we focus on invertible dynamics on the circle, more specifically on the study of
so-called \emph{multicritical circle maps\/}. By a multicritical circle map, we mean a reasonably smooth
orientation-preserving homeomorphism $f: S^1\to S^1$ having a finite number of critical points, all of which are
assumed to be {\it non-flat\/} (of {\it power-law\/} type, see Definition \ref{defmccm}). If $f$ has only one critical point, we sometimes say that $f$ is a {\it unicritical\/} circle map. In the present paper, by ``reasonably smooth'' we mean that
$f$ is at least $C^3$; this degree of smoothness allows us to use certain tools -- such as the so-called Yoccoz inequality, see Section \ref{sectools} -- which under current technology can only be established with the help of the {\it Schwarzian derivative\/} of $f$.

As it happens, rigidity can only be attained in the absence
of periodic points. Thus we assume throughout that $\mathrm{Per}(f)=\textrm{\O}$, which is tantamount to saying that $f$ has
{\it irrational rotation number\/}. A fundamental theorem due to Yoccoz \cite{yoccoz} states that every such multicritical circle map is topologically conjugate to the rigid rotation having the same rotation number. In particular, a topological conjugacy always exists between any two multicritical circle maps with the same (irrational)
rotation number.
The relevant rigidity questions here are, thus: (1) When is such conjugacy $C^1$? (2) Can this conjugacy be better than $C^1$?
Note that a necessary condition for these questions to be well-posed is that the conjugacy establishes a one-to-one
correspondence between the critical points of one map and the critical points of the other. Another necessary condition for $C^1$
rigidity is that the criticalities (or power-law exponents) of corresponding critical points under the conjugacy be equal.
It is conjectured that these necessary conditions are also sufficient (see \cite{EdF}). In the unicritical case, these questions
have been fully answered, thanks to the combined efforts of a number of mathematicians -- see \cite{EdsonThesis}, \cite{edsonETDS},
\cite{edsonwelington1}, \cite{edsonwelington2}, \cite{PabloThesis}, \cite{GMdM2015}, \cite{GdM2013}, \cite{khaninteplinsky},
\cite{khmelevyampolsky}, \cite{yampolsky1}, \cite{yampolsky2}, \cite{yampolsky3}, \cite{yampolsky4}. We summarize these contributions in the following statements: on one hand, any two $C^3$ circle homeomorphisms with the same irrational
rotation number of \emph{bounded type} and with a single critical point (of the same odd power-law type) are conjugate to each other
by a $C^{1+\alpha}$ circle diffeomorphism, for some universal $\alpha>0$ \cite{GdM2013}. On the other hand, any two $C^4$ circle
homeomorphisms with the same irrational rotation number and with a unique critical point (again, of the same odd type), are conjugate
to each other by a $C^1$ diffeomorphism \cite{GMdM2015}. Moreover, this conjugacy is a $C^{1+\alpha}$ diffeomorphism for a certain
set of rotation numbers that has full Lebesgue measure (see \cite[Section 4.4]{edsonwelington1} for its definition), but {\it does not\/} include all irrational rotation numbers (see the counterexamples in \cite{avila} and \cite[Section 5]{edsonwelington1}).

\subsubsection*{Quasisymmetry} As it turns out, an important first step towards rigidity is what is known as {\it quasisymmetric rigidity\/}. In the recent paper \cite{CvS}, this step is accomplished with the use of complex-analytic techniques in a fairly general context covering multimodal maps of the interval or the circle. For multicritical circle maps, this step was accomplished by purely real methods in \cite{EdF}. In that paper, it was proved that if $f$ and $g$ are two $C^3$ multicritical circle maps with the same irrational rotation number, and if there is a conjugacy between $f$ and $g$ that maps the critical points of $f$ to the critical points of $g$, then such conjugacy is a {\it quasisymmetric homeomorphism\/} (even if the criticalities at corresponding critical points are not the same!). A self-homeomorphism $h$ of the line $\mathbb{R}$ or the circle $S^1=\mathbb{R}/\mathbb{Z}$
is {\it quasisymmetric\/} if there exists a constant $M>1$ such that
\[
 \frac{1}{M}\leq \frac{h(x+t)-h(x)}{h(x)-h(x-t)} \leq M \ ,
\]
for all $x$ on the line or circle and all $t>0$.

\bigskip

\begin{figure}[h!]
\adjustbox{scale=0.8,center}{
\begin{tikzcd}[arrows=Rightarrow
]
&\mathrm{C^1\,\,diffeomorphism}\arrow{d}&\\
&\mathrm{Bi-Lipschitz} \arrow{dr}{} \arrow{dl}{} & \\
\mathrm{Q\,S} \arrow{d}&   &  \mathrm{Abs. \,\,Cont.} \arrow{d}\\%
\mathrm{Bi-H\ddot{o}lder}   &   &\mathrm{Bounded \,\,Variation}
\end{tikzcd}
}
\caption[diagrama]{\label{diagrama} Hierarchy involving \emph{quasisymmetry} and other classical notions of continuity for circle homeomorphisms.}
\end{figure}
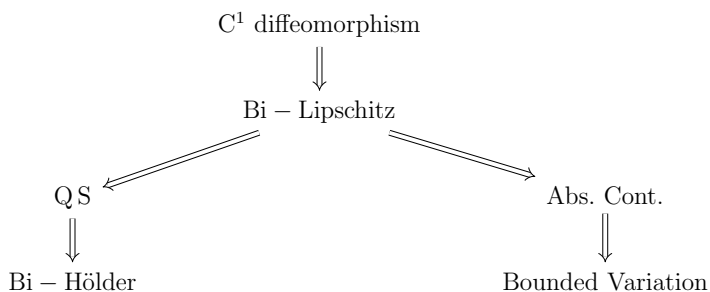

Quasisymmetry, as defined above, can be regarded as a very weak form of (geometric) regularity (see Figure \ref{diagrama}).
It is in fact so weak that one might guess that {\it any\/} conjugacy between $f$ and $g$ above will be quasisymmetric.
This guess is reinforced by a theorem due to Herman and \'Swi\c{a}tek (see \cite{H} and \cite{swiatek}) according to which
every multicritical circle map whose rotation number is an irrational of {\it bounded type\/} is quasisymmetrically conjugate to the corresponding rotation. Note, in particular, that such quasisymmetric conjugacies identify critical points with regular points.

However, the above guess is unfortunately wrong. Our purpose in the present paper is to show that a conjugacy between two critical circle maps is {\it almost never quasisymmetric\/}. The first goal is to identify a \emph{mechanism} which forces the breakdown of quasisymmetry for a topological conjugacy (see Lemma \ref{crossratiodistortion} in Section \ref{StheproofofA}). The second goal is to show that the above mechanism is \emph{abundant}, both from the topological and measure-theoretical viewpoints (see Theorem \ref{propresiduais} in Section \ref{sechomeos}). The precise statements of our results will be given below -- see Theorems \ref{ThmA}, \ref{ThmB} and \ref{ThmC}.

\subsubsection*{Orbit-flexibility} Some of our results can be stated in the light of the complementary concepts of {\it orbit-rigidity\/} and {\it orbit-flexibility\/}, which we
presently describe. We say that a minimal circle homeomorphism $f$ is {\it quasisymmetrically orbit-rigid\/} if for any pair of points
$x,y$ on the circle there exists a quasisymmetric homeomorphism $h_{x,y}$ which conjugates $f$ to itself and maps $x$ to $y$.
If $f$ is {\it not\/} quasisymmetrically orbit-rigid, we say that $f$ is {\it quasisymmetrically orbit-flexible\/}.
Thus, irrational rotations and sufficiently smooth circle diffeomorphisms with Diophantine rotation numbers are quasisymmetrically orbit-rigid.
Likewise, by the above-mentioned Herman-\'Swi\c{a}tek theorem, multicritical circle maps with rotation number of
{\it bounded\/} type are quasisymmetrically orbit-rigid.
By contrast, we will show in Theorem \ref{ThmA} that (uni)critical circle maps whose rotation numbers belong to a certain full-measure set are quasisymmetrically orbit-flexible (see also Proposition \ref{propuncount}). In particular, the centralizers of such maps in the group of all homeomorphisms of the circle contain non-quasisymmetric elements (see Section \ref{subsecflex} below).

\subsection{Statement for unicritical maps}\label{secenunciadoThmA} In the \emph{unicritical} case we have the following \emph{coexistence} phenomenon.

\begin{maintheorem}[Coexistence]\label{ThmA} There exists a full Lebesgue measure set $\bm{R}_A \subset [0,1]$ of irrational numbers with the following property: let $f$ and $g$ be two $C^{3}$ circle homeomorphisms with a single (non-flat) critical point (say, $c_f$ and $c_g$ respectively) and with $\rho(f)=\rho(g)\in\bm{R}_A$. For any given $x \in S^1$ let $h_x\in\mathrm{Homeo}^+(S^1)$ be the topological conjugacy between $f$ and $g$ determined by $h_x(x)=c_g$. Let $\mathcal{A}$ be the set of points $x \in S^1$ such that the homeomorphism $h_x$ is quasisymmetric, and let $\mathcal{B}=S^1\setminus\mathcal{A}$ be its complement in the unit circle (that is, $\mathcal{B}$ is the set of points $x \in S^1$ such that the homeomorphism $h_x$ is \emph{not} quasisymmetric). Then $\mathcal{A}$ is dense in $S^1$, while $\mathcal{B}$ contains a residual set (in the sense of Baire) and it has full $\mu_f$-measure, where $\mu_f$ denotes the unique $f$-invariant probability measure.
\end{maintheorem}

\begin{remark} A somewhat related coexistence phenomenon occurs in the context of Lorenz maps, and also in the context of circle maps with flat intervals (see \cite{MPW} and references therein).
\end{remark}

\begin{remark} The proof of Theorem \ref{ThmA}, to be given in \S\ref{StheproofofA}, still works if \emph{one of the two} maps has more than one critical point.
\end{remark}

Let us pose two questions that arise from Theorem \ref{ThmA}.

\begin{question} Denote by $\bt\subset(0,1)$ the set of irrational numbers of \emph{bounded type}. As already mentioned, a theorem of Herman \cite{H} implies that $\bm{R}_A$ is disjoint from $\bt$ (since in this case \emph{all} conjugacies are quasisymmetric, see Section \ref{subsecflex} below). Is it true that $\bm{R}_A=[0,1]\setminus(\Q\cup\bt)$? Is it true, at least, that $\bm{R}_A$ contains a residual subset of $[0,1]$?
\end{question}

\begin{question}\label{pergconjA} Note that both sets $\mathcal{A}$ and $\mathcal{B}$ defined in Theorem \ref{ThmA} are $f$-invariant. Indeed, this follows from the identity $h_x=h_{f(x)} \circ f$ and the fact that $f$ itself (hence $f^n$ for all $n\in \mathbb{Z}$) is a quasisymmetric homeomorphism. As explained above, the critical point of $f$ belongs to $\mathcal{A}$ (and then its whole orbit), since $h_{c_f}$ is always a quasisymmetric homeomorphism (this was proved by Yoccoz in an unpublished work, see \cite[Corollary 4.6]{edsonwelington1}). It could be the case that $\mathcal{A}=\big\{f^n(c_f):n\in\Z\big\}$. Is it true, at least, that $\mathcal{A}$ is a countable set?
\end{question}

In Section \ref{subsecflex} below we describe more precisely the notion of orbit-flexibility, and state some straightforward consequences of Theorem \ref{ThmA}. In Section \ref{secboundedgeo} we state some further consequences of Theorem \ref{ThmA}, this time involving geometric bounds for dynamical partitions (see Theorem \ref{ThmD}).

\subsection{Statements for multicritical maps} Given an irrational number $\rho$, we denote by $a_n=a_n(\rho)$, $n\in \mathbb{N}$, the sequence of its partial quotients (see Section \ref{sechomeos}). Let us consider the set $\mathbb{E}_\infty$ consisting of all numbers $\rho\in (0,1)$ for which the corresponding $a_n$'s are {\it even\/} and $\lim_{n\to\infty}{a_n}\;=\;\infty$. It is easy to see that $\mathbb{E}_\infty$ is a meager set whose Lebesgue measure is equal to zero. Despite being both topologically and measure-theoretically negligible, this set does contain some interesting Diophantine, Liouville and transcendental numbers, see Section \ref{sec:eventype}. Our second goal in the present paper is to prove the following result.

\begin{maintheorem}\label{ThmB} There exists a set\, $\mathcal{G}\subset[0,1]^2$, which contains a residual set (in the Baire sense) and has full Lebesgue measure, for which the following holds. Let $f$ and $g$ be two $C^3$ multicritical circle maps with the same irrational rotation number $\rho$ and such that the map $f$ has exactly one critical point $c_0$, whereas the map $g$ has exactly two critical points $c_1$ and $c_2$. Denote by $\alpha$ and $1-\alpha$ the $\mu_g$-measures of the two connected components of $S^1\setminus\{c_1,c_2\}$, where $\mu_g$ denotes the unique invariant probability measure of $g$. If $(\rho,\alpha)$ belongs to $\mathcal{G}$, then the topological conjugacy between $f$ and $g$ that takes $c_0$ to $c_1$ is not quasisymmetric. Moreover, the set of rotation numbers $\bm{R}_B=\{\rho:\,(\rho,\alpha)\in \mathcal{G}\ \textrm{for some}\ \alpha\}$ contains the set $\mathbb{E}_\infty$ defined above.
\end{maintheorem}

The proofs of both Theorem \ref{ThmA} and Theorem \ref{ThmB} will be given in Section \ref{StheproofofA}. In Section \ref{S:real} we will prove the following auxiliary result.

\medskip

\noindent{\bf The $C^{\infty}$ Realization Lemma.}\, {\it For any given $(\rho,\alpha)\in\big([0,1]\setminus\Q\big)\times(0,1)$ there exists a $C^{\infty}$ bi-critical circle map with rotation number $\rho$, a unique invariant Borel probability measure $\mu$ and
with exactly two critical points $c_1$ and $c_2$ such that the two connected components
of $S^1\setminus\{c_1,c_2\}$ have $\mu$-measures equal to $\alpha$ and $1-\alpha$ respectively.}

\begin{remark} It is possible to prove a similar \emph{Analytic Realization Lemma} using the results of Zakeri in \cite[Section 7]{zak}.
\end{remark}

Together with Theorem \ref{ThmB}, the $C^{\infty}$ Realization Lemma implies our third main result.

\begin{maintheorem}\label{ThmC} There exists a set\, $\bm{R}_C \subset [0,1]$ of irrational numbers, which contains a residual set (in the Baire sense), has full Lebesgue measure and contains $\mathbb{E}_\infty$, for which the following holds.
For each $\rho\in\bm{R}_C$, there exist two $C^{\infty}$ multicritical circle maps $f, g: S^1\to S^1$ with the following properties:
\begin{enumerate}
 \item Both maps have the same rotation number $\rho$;
 \item The map $f$ has exactly one critical point $c_0$, whereas the map $g$ has exactly two critical points $c_1$ and $c_2$;
 \item The topological conjugacy between $f$ and $g$ that takes $c_0$ to $c_1$ is not quasisymmetric.
\end{enumerate}
\end{maintheorem}

\subsection{Quasisymmetric orbit-flexibility of critical circle maps}\label{subsecflex} Following Yoccoz \cite{yoccoz2,yoccoz22}, we denote by $Z_0(f)=\{h\in \mathrm{Homeo}^+(S^1):\, h\circ f= f\circ h\}$ the \emph{centralizer}
of $f$ in $\mathrm{Homeo}^+(S^1)$. We also denote by $\mathrm{QS}(S^1)$ the subgroup of  $\mathrm{Homeo}^+(S^1)$ consisting of
those homeomorphisms of the circle that are quasisymmetric. In this language, Theorem \ref{ThmA} has the following immediate consequence (see also \cite[Section 4]{AChE20} for recent results on the centralizers of some analytic circle maps).

\begin{customthm}{} If $f:S^1\to S^1$ is a unicritical circle map with $\rho(f)\in\bm{R}_A$, then $f$ is quasisymmetrically orbit-flexible.
 In particular, $Z_0(f)\setminus \mathrm{QS}(S^1)\neq \mathrm{\O}$.
\end{customthm}

In fact, much more can be obtained from Theorem \ref{ThmA}. First, we need a definition. Let $f:S^1\to S^1$ be a minimal circle homeomorphism.

\begin{definition}\label{defgeomeq} If $x,y\in S^1$, we say that $x$ is $f$-equivalent to $y$, and write $x\sim_f y$, if there exists a quasisymmetric homeomorphism $h\in Z_0(f)$ such that $h(x)=y$.
\end{definition}

It is clear that $\sim_f$ is an equivalence relation, so we can consider the set of equivalence classes $X_f=S^1/\sim_f$. Below, during the proof of Proposition \ref{propuncount}, we will use the following observation.

\begin{lemma}\label{leminhaclases} All equivalence classes are homeomorphic to each other.
\end{lemma}

\begin{proof}[Proof of Lemma \ref{leminhaclases}] Let us mark some point $c \in S^1$. For any given $x \in S^1$ consider $F_x:S^1 \to S^1$ defined as follows: given $y \in S^1$ let $h_{x,y} \in Z_0(f)$ be determined by $h_{x,y}(x)=y$, and then let $F_x(y)$ be defined by $h_{x,y}\big(F_x(y)\big)=c$. It not difficult to prove that $F_x$ is a circle homeomorphism which identifies the class of $x$ with the one of $c$. In particular, given $x,y \in S^1$, the homeomorphism $F_y^{-1} \circ F_x$ identifies the class of $x$ with the class of $y$.
\end{proof}

Note that if $f$ is either a diffeomorphism or a ($C^3$) multicritical circle map, then points in the same $f$-orbit are $f$-equivalent. More generally, for such $f$'s, if $x\sim_f y$ then for each $x'\in \mathcal{O}_f(x)$ and each $y'\in \mathcal{O}_f(y)$ we have $x'\sim_f y'$. This happens because, in the cases considered, $f$ itself (hence $f^n$ for all $n\in \mathbb{Z}$) is a quasisymmetric homeomorphism. Note that, being $f$-invariant, all equivalence classes are dense in the unit circle.

In the language introduced before, if $X_f$ reduces to a single point, then $f$ is quasisymetrically orbit-rigid, whereas if $X_f$ has more than one point, then $f$ is quasisymetrically orbit-flexible. Now we can state the following simple consequence of Theorem \ref{ThmA}. 

\begin{prop}\label{propuncount} If $f:S^1\to S^1$ is a unicritical circle map whose rotation number belongs to the set $\bm{R}_A$ of Theorem \ref{ThmA}, then all its equivalence classes are meagre (in the sense of Baire). In particular $X_f$ is uncountable. 
\end{prop}

\begin{proof}[Proof of Proposition \ref{propuncount}] By definition, the set $\mathcal{A}$ given by Theorem \ref{ThmA} (applied to the particular case $g=f$) is the equivalence class of $c_f$, the critical point of $f$. Being disjoint from the residual set $\mathcal{B}$, the set $\mathcal{A}$ is meagre. By Lemma \ref{leminhaclases}, all classes are meagre, and therefore, by Baire's theorem, their number is uncountable.
\end{proof}

By contrast, if $f: S^1\to S^1$ is a {\it smooth diffeomorphism\/} whose rotation number is Diophantine, then by a well-known theorem due to Herman and Yoccoz \cite{hermanIHES,yoccoz3}, $f$ is $C^1$ conjugate (in fact smoothly conjugate) to a rotation, and this immediately implies that $X_f$ is a single point. As mentioned before, the same happens with any irrational rotation or with any multicritical circle map with rotation number of bounded type. Indeed, as it follows from a result of Herman \cite{H}, any multicritical circle map $f$ with irrational rotation number $\rho$ of bounded type is quasisymmetrically conjugate to the rotation of angle $\rho$ (denoted by $R_{\rho}$): there exists a quasisymmetric circle homeomorphism $h$ such that $h \circ f=R_{\rho} \circ h$. Now mark some point $x \in S^1$ and for any given $y \in S^1$ consider the angle $\theta_y$ between $h(x)$ and $h(y)$, that is: $R_{\theta_y}\big(h(x)\big)=h(y)$. Then the homeomorphism $h_{x,y}=h^{-1} \circ R_{\theta_y} \circ h$ is quasisymmetric, commutes with $f$ (because $R_{\theta_y}$ commutes with $R_{\rho}$) and identifies $x$ with $y$. In other words, $x \sim_f y$ and then $X_f$ is a single point.

\subsection{Unbounded geometry}\label{secboundedgeo} Let $f$ be a $C^3$ multicritical circle map with irrational rotation number. We say that $f$ has \emph{bounded geometry} at $x \in S^1$ if there exists $K>1$ such that for all $n\in\nt$ and for every pair $I,J$ of adjacent atoms of $\mathcal{P}_n(x)$ we have$$K^{-1}\,|I| \leq |J| \leq K\,|I|\,,$$where $\big\{\mathcal{P}_n(x)\big\}_{n\in\nt}$\, is the standard sequence of dynamical partitions of the circle associated to $x \in S^1$ (see Section \ref{subseccombpart}). With this at hand, consider the set$$\mathcal{A}=\mathcal{A}(f)=\{x \in S^1:\,\mbox{$f$ has bounded geometry at $x$}\}\,.$$The relation between bounded geometry and quasisymmetric homeomorphisms is given by the following result.

\begin{prop}\label{propappeq} Let $f$ be a multicritical circle map with irrational rotation number, and let $x\in\mathcal{A}(f)$. As before, for any given $y \in S^1$ let $h_{x,y} \in Z_0(f)$ be determined by $h_{x,y}(x)=y$. Then$$y\in\mathcal{A}(f) \Leftrightarrow h_{x,y} \in \mathrm{QS}(S^1)\,.$$
\end{prop}

\begin{proof}[Proof of Proposition \ref{propappeq}] For the ``if" implication suppose, by contradiction, that $y\notin\mathcal{A}$. This means that there exists a sequence $\{n_k\}_{k\in\nt}\subset\nt$ such that for each $k\in\nt$ we can find a pair $I_k,J_k$ of adjacent atoms of $\mathcal{P}_{n_k}(y)$ satisfying $\lim_{k}|I_k|/|J_k|=+\infty$\,. However, both intervals $h_{x,y}^{-1}(I_k)$ and $h_{x,y}^{-1}(J_k)$ are adjacent and belong to $\mathcal{P}_{n_k}(x)$, and since $x\in\mathcal{A}$, the ratios $\big|h_{x,y}^{-1}(I_k)\big|/\big|h_{x,y}^{-1}(J_k)\big|$ are bounded. But this is impossible since, being a quasisymmetric homeomorphism, $h_{x,y}$ is bi-H\"older (recall Figure \ref{diagrama}). For the ``only if" implication we refer the reader to \cite[Sections 5.1 and 5.2]{EdF}.
\end{proof}

An immediate consequence of Proposition \ref{propappeq} is that the set $\mathcal{A}$ is $f$-invariant, since $f$ itself (hence $f^n$ for all $n\in \mathbb{Z}$) is a quasisymmetric homeomorphism. As it follows from the classical real bounds of Herman and \'Swi\c{a}tek (see Theorem \ref{realbounds} for its precise statement), all critical points of $f$ belong to $\mathcal{A}$. Being $f$-invariant and non-empty, the set $\mathcal{A}$ is dense in the unit circle. However, the following consequence of Theorem \ref{ThmA} shows that $\mathcal{A}$ can be rather small.

\begin{maintheorem}\label{ThmD} Let $\bm{R}_A\subset(0,1)$ be the full Lebesgue measure set given by Theorem \ref{ThmA}, and let $f$ be a $C^3$ critical circle map with a single (non-flat) critical point and rotation number $\rho\in\bm{R}_A$. Then the set $\mathcal{A}(f)$ is meagre (in the sense of Baire) and it has zero $\mu_f$-measure.
\end{maintheorem}

To prove Theorem \ref{ThmD} note first that Proposition \ref{propappeq} is saying that the set $\mathcal{A}$ is an equivalence class for the $\sim_f$\, relation and then, by Proposition \ref{propuncount}, we already know that it is meagre. Moreover, since the critical point of $f$ belongs to $\mathcal{A}$ (again, see Theorem \ref{realbounds}), we deduce that $\mathcal{A}$ is precisely the equivalence class of the critical point. With this at hand, Theorem \ref{ThmD} follows at once from Theorem \ref{ThmA} just by considering the particular case $g=f$.

By contrast, recall that if $f$ has bounded combinatorics, then the set $\mathcal{A}(f)$ is the whole circle (as already discussed at the end of Section \ref{subsecflex})\,: $f$ has bounded geometry at any point.

\subsection{Brief summary} Here is how the paper is organized. In \S \ref{sechomeos} we recall some basic facts concerning circle homeomorphisms without periodic points, as well as some standard tools for the study of multicritical circle maps -- the most
important for us being cross-ratio distortion and Yoccoz's Lemma. We also introduce the concepts of {\it renormalization ancestors\/}
and {\it renormalization trails\/}, which are specific to the present paper, and we state Theorem \ref{propresiduais}, a key
result for our purposes. In \S \ref{Secskew} we introduce a certain skew product whose ergodicity and topological exactness will
be crucial in proving Theorem \ref{propresiduais}. The proof of Theorem \ref{propresiduais} is given in \S \ref{provaproptrails}.
In \S \ref{sec:eventype}, we examine the class $\mathbb{E}_\infty$ of rotation numbers that appears in the statement of Theorem \ref{ThmB} above. The proofs of Theorems \ref{ThmA} and \ref{ThmB}  will be given in \S \ref{StheproofofA}. The auxiliary concept of {\it admissible pairs\/} for bi-critical circle maps is introduced in \S \ref{S:real}, where we prove the $C^{\infty}$ Realization Lemma stated above (and recall that, when combined with Theorem \ref{ThmB}, the $C^{\infty}$ Realization Lemma implies Theorem \ref{ThmC}). The paper closes with Appendix \ref{apperg}, dedicated to the proof of the ergodicity of the skew product introduced in \S \ref{Secskew}, and  Appendix \ref{appren}, which contains some informal remarks on the connection of said skew-product with renormalization theory. 

\section{Minimal circle homeomorphisms}\label{sechomeos}

\subsection{Combinatorics and partitions}\label{subseccombpart} Let $f:S^1 \to S^1$ be an orientation preserving circle
homeomorphism with irrational rotation number $\rho$. As it is well know, $\rho$ has an infinite \emph{continued fraction expansion}, say
\begin{equation*}
      \rho(f)= [a_{0} , a_{1} , \cdots ]=
      \cfrac{1}{a_{0}+\cfrac{1}{a_{1}+\cfrac{1}{ \ddots} }} \ .
    \end{equation*}

A classical reference for continued fraction expansion is the monograph \cite{khin}. Truncating the expansion at level $n-1$,
we obtain a sequence of fractions $p_n/q_n$, which are called the \emph{convergents} of the irrational $\rho$.
$$
\frac{p_n}{q_n}\;=\;[a_0,a_1, \cdots ,a_{n-1}]\;=\;\dfrac{1}{a_0+\dfrac{1}{a_1+\dfrac{1}{\ddots\dfrac{1}{a_{n-1}}}}}\ .
$$

The sequence of denominators $q_n$, which we call the \emph{return times}, satisfies
\begin{equation*}
 q_{0}=1, \hspace{0.4cm} q_{1}=a_{0}, \hspace{0.4cm} q_{n+1}=a_{n}\,q_{n}+q_{n-1} \hspace{0.3cm} \text{for $n \geq 1$} .
\end{equation*}

Since $\rho$ is irrational, $f$ admits a unique invariant Borel probability measure $\mu$. Assuming that $f$ has no wandering
intervals, we deduce that there exists a circle homeomorphism $h:S^1 \to S^1$ which is a topological conjugacy between $f$ and
the rigid rotation by angle $\rho(f)$, that we denote by $R_{\rho(f)}$. More precisely, the following diagram commutes:
$$
\begin{CD}
(S^1,\mu)@>{f}>>(S^1,\mu)\\
@V{h}VV             @VV{h}V\\
{(S^1,\lambda)}@>>{R_{\rho}}>{(S^1,\lambda)}
\end{CD}
$$
where $\lambda$ denotes the normalized Lebesgue measure in the unit circle (the Haar measure for the multiplicative group of
complex numbers of modulus $1$). Therefore $\mu$ is just the push-forward of the Lebesgue measure under $h^{-1}$, that is,
$\mu(A)=\lambda\big(h(A)\big)$ for any Borel set $A$ in the unit circle (recall that the conjugacy $h$ is unique up to
post-composition with rotations, so the measure $\mu$ is well-defined). In particular, $\mu$ has no atoms and gives
positive measure to any open set (for more information on the measure $\mu$ see \cite{dFG2016, dFG20, tru} and references therein).

We consider now intervals of the form $(\,\cdot\,,\,\cdot\,]$\,, that is, \emph{open on its left} and \emph{closed on its right}.
For each non-negative integer $n$, let $I_{n}$ be the interval with endpoints $x$ and $f^{q_n}(x)$ containing  $f^{q_{n+2}}(x)$, namely,
$I_n=\big(x,f^{q_n}(x)\big]$ and $I_{n+1}=\big(f^{q_{n+1}}(x),x\big]$.

As it is well known, for each $n\geq 0$, the collection of intervals$$\mathcal{P}_n(x)\;=\; \big\{f^i(I_n):\;0\leq i\leq q_{n+1}-1\big\}\cup\big\{f^j(I_{n+1}):\;0\leq j\leq q_{n}-1\big\}$$is a partition of the unit circle (see for instance the appendix in \cite{EdFG}), called the {\it $n$-th dynamical partition\/} associated to $x$. The intervals of the form $f^i(I_n)$ are called \emph{long}, whereas those of the form $f^j(I_{n+1})$ are called \emph{short}. The initial partition $\mathcal{P}_0(x)$ is given by$$\mathcal{P}_0(x)=\left\{\big(f^{i}(x),f^{i+1}(x)\big]:\,i\in\{0,...,a_0-1\}\right\}\cup\big\{\big(f^{a_{0}}(x),x\big]\big\}.$$

Let us now give the formal definition of a \emph{multicritical circle map}, the main object of study in the present paper.

\begin{definition}\label{defmccm} A critical point $c$ of a one-dimensional $C^3$ map $f$ is said to be \emph{non-flat} of criticality $d>1$ if there exists a neighbourhood $W$ of $c$ such that $f(x)=f(c)+\phi(x)\big|\phi(x)\big|^{d-1}$ for all $x \in W$, where $\phi : W \rightarrow \phi(W)$ is a $C^{3}$ diffeomorphism satisfying $\phi(c)=0$. A \emph{multicritical circle map} is an orientation preserving $C^3$ circle homeomorphism having $N \geq 1$ critical points, all of which are non-flat.
\end{definition}

Throughout this paper we make no further assumption on the criticality of any critical point. The following fundamental result was obtained by Herman and \'Swi\c{a}tek in the eighties \cite{H,swiatek}.

\begin{theorem}[The real bounds]\label{realbounds} Given $N\geq 1$ in $\nt$ and $d>1$ there exists a universal constant $C=C(N,d)>1$ with the following property: for any given multicritical circle map $f$ with irrational rotation number, and with at most $N$ critical points whose criticalities are bounded by $d$, there exists $n_0=n_0(f)\in\nt$ such that for each critical point $c$ of $f$, for all $n \geq n_0$, and for every pair $I,J$ of adjacent atoms of $\mathcal{P}_n(c)$ we have:$$C^{-1}\,|I| \leq |J| \leq C\,|I|\,,$$where $|I|$ denotes the Euclidean length of an interval $I$.
\end{theorem}

In the language introduced in Section \ref{secboundedgeo}, Theorem \ref{realbounds} is saying that a multicritical circle map has bounded geometry at any of its critical points. A detailed proof of Theorem \ref{realbounds} can also be found in \cite{EdF,EdFG}.

\subsection{The Gauss map}\label{gaussmap}

For any real number $x$ denote by $\lfloor x\rfloor$ the \emph{integer part} of $x$, that is,
the greatest integer less than or equal to $x$. Also, denote by $\{x\}$ the \emph{fractional part} of $x$, that is, $\{x\}=x-\lfloor x\rfloor\in[0,1).$

Recall that the \emph{Gauss map} $G:[0,1]\to[0,1]$ is given by
\[
G(\rho)=\left\{\frac{1}{\rho}\right\}\mbox{ for $\rho\neq 0$\,, and $G(0)=0$.}
\]
Both $\Q\cap[0,1]$ and $[0,1]\setminus\Q$ are $G$-invariant. Under the action of $G$, all rational numbers in $[0,1]$ eventually
land on the fixed point at the origin, while the irrationals remain in the union $\bigcup_{k \geq 1}\left(\frac{1}{k+1},\frac{1}{k}\right)$. Moreover, for any $\rho\in(0,1)\setminus\Q$ and any $j\in\nt$ we have that $G^j(\rho)\in\left(\frac{1}{k+1},\frac{1}{k}\right)$ if, and only if, $a_j=k$, where $a_j$ denotes the partial quotients of $\rho$ (just as in Section \ref{subseccombpart} above). Indeed, if $\rho= [a_{0} , a_{1} , a_{2} , \cdots ]$ belongs to $\big(1/(k+1),1/k\big)$, then $1/\rho=a_0+[a_{1} , a_{2} , \cdots ]$ and then $a_0=\left\lfloor\frac{1}{\rho}\right\rfloor=k$ and $G(\rho)=[a_{1} , a_{2} , \cdots ]$. In particular, the Gauss map acts as a \emph{shift} on the continued fraction expansion of $\rho$.

\medskip

As it is well known, the map $G$ preserves an ergodic Borel probability measure $\nu$ (called the \emph{Gauss measure}) given by:$$\nu(A)=\frac{1}{\log 2}\int_{A}\frac{d\rho}{1+\rho}\quad\mbox{for any Borel set $A \subset [0,1]$.}$$

In particular, the Gauss measure $\nu$ is equivalent to the Lebesgue measure on $[0,1]$ ({\it i.e.,\/} they share the same null sets). In Section \ref{provaproptrails}, during the proof of Lemma \ref{leminhacomb}, we will make repeated use of the following well-known formula.

\begin{lemma}\label{formmuIn} Let $f:S^1 \to S^1$ be an orientation preserving circle
homeomorphism with irrational rotation number $\rho$, and with unique invariant measure $\mu$. For any $x \in S^1$ and any $n\in\nt$ we have:
\begin{equation}\label{formuInG}
\mu(I_n)=\prod_{j=0}^{j=n}G^j(\rho)=\rho\,G(\rho)\,G^2(\rho)\,\cdots\,G^n(\rho),
\end{equation}where $I_{n}$ is the interval with endpoints $x$ and $f^{q_n}(x)$ containing $f^{q_{n+2}}(x)$, as defined in Section \ref{subseccombpart}.
\end{lemma}

Note, in particular, that\, $\mu(I_{n+1})=G^{n+1}(\rho)\,\mu(I_n)$\, for all $n\in\nt$.

\begin{proof}[Proof of Lemma \ref{formmuIn}] The proof goes by induction on $n\in\nt$. First note that, since $I_0=\big(x,f(x)\big]$ is a \emph{fundamental domain} for $f$, we have $\mu(I_0)=\rho$. Now let $a_0\in\nt$ be defined by:$$a_0\,\mu(I_0) \leq \mu(S^1) < (a_0+1)\,\mu(I_0)\,.$$
In other words, $a_0\,\rho \leq 1 < (a_0+1)\,\rho$.\, This implies $a_0 \leq \frac{1}{\rho} < a_0+1$, and then $a_0=\left\lfloor\frac{1}{\rho}\right\rfloor$. In particular, $\mu(I_1)=\mu(S^1)-a_0\,\mu(I_0)=1-a_0\,\rho=(\frac{1}{\rho}-a_0)\,\rho=\left\{\frac{1}{\rho}\right\}\rho=\rho\,G(\rho)$. This shows that \eqref{formuInG} holds for $n=0$ and $n=1$. Now fix some $n\in\nt$ and let $a_{n+1}\in\nt$ be defined by:$$a_{n+1}\,\mu(I_{n+1}) \leq \mu(I_{n}) < (a_{n+1}+1)\,\mu(I_{n+1})\,.$$
In other words, $\displaystyle a_{n+1}=\left\lfloor\frac{\mu(I_{n})}{\mu(I_{n+1})}\right\rfloor$. Assuming that \eqref{formuInG} holds for $n$ and $n+1$, we obtain $\displaystyle a_{n+1}=\left\lfloor\frac{1}{G^{n+1}(\rho)}\right\rfloor$ and then:
\begin{align*}
\mu(I_{n+2})&=\mu(I_n)-a_{n+1}\,\mu(I_{n+1})=\left(\frac{1}{G^{n+1}(\rho)}-a_{n+1}\right)\prod_{j=0}^{n+1}G^j(\rho)=\\
&=\left\{\frac{1}{G^{n+1}(\rho)}\right\}\prod_{j=0}^{n+1}G^j(\rho)=G\big(G^{n+1}(\rho)\big)\prod_{j=0}^{n+1}G^j(\rho)=\prod_{j=0}^{n+2}G^j(\rho)\,.
\end{align*}
This implies that \eqref{formuInG} holds for all $n\in\nt$.
\end{proof}

\subsection{Renormalization trails and ancestors}\label{sectrails} Consider the rectangle $R=[0,1]\times[-1,1]$ in $\R^2$, and let $M=\big([0,1]\setminus\Q\big)\times[-1,1] \subset R$. Recall, from Section \ref{subseccombpart}, that $f:S^1 \to S^1$ denotes an orientation preserving circle homeomorphism with irrational rotation number $\rho=[a_0,a_1,a_2,...]$. Let us fix some point $x$ in the unit circle. For any given $y$ in $S^1$, we will define/construct in this section a sequence of pairs $(\rho_n,\alpha_n) \in M$, called renormalization trail (see Definition \ref{deftrails} below) of $y$ with respect to $x$ and $f$. Let us define simultaneously the initial cases $n=0$ and $n=1$. First, let $\rho_0=\rho=[a_0,a_1,a_2,...]\in[0,1]\setminus\Q$ and $\rho_1=G(\rho)=[a_1,a_2,...]\in[0,1]\setminus\Q$. To define $\alpha_0$ and $\alpha_1$ consider both intervals
$$I_0=\big(x,f(x)\big]\quad\mbox{and}\quad
I_1=\big(f^{a_0}(x),x\big]$$as defined in Section \ref{subseccombpart}. If $y$ belongs to the short interval $I_1$ we define:$$\alpha_0=\mu\big((x,y)\big)\in[0,1-a_0\,\rho_0]\quad\mbox{and}\quad
\alpha_1=-\,\frac{\mu\big((x,y)\big)}{\mu(I_{1})}\,\in[-1,0].$$

Otherwise, there exist $y_0$ in the long interval $I_0$ and $i\in\{0,1,...,a_0-1\}$ such that $f^i(y_0)=y$,
in which case we define:$$\alpha_0=1-\big[\mu\big((x,y_0)\big)+i\,\rho_0\big]=1-\mu\big((x,y)\big)\in[1-a_0\,\rho_0,1]\quad\mbox{and}\quad
\alpha_1=\frac{\mu\big((x,y_0)\big)}{\mu(I_0)}\,\in[0,1].$$

Note that, in the definition of $\alpha_0$, we are measuring in the \emph{counterclockwise} sense: in the first case, we measure $\mu\big((x,y)\big)$ considering the arc determined by $x$ and $y$ which is contained in $I_1$, while in the second case we measure $\mu\big((x,y_0)\big)$ considering the arc determined by $x$ and $y_0$ which is contained in $I_0$. In this way we obtain the first two terms of the sequence of pairs
$(\rho_n,\alpha_n) \in M=\big([0,1]\setminus\Q\big)\times[-1,1]$. After the first $n$ terms are defined, let $\rho_{n+1}\in[0,1]\setminus\Q$ be given by$$\rho_{n+1}=G^{n+1}(\rho)=G^{n+1}\big([a_0,a_1,...]\big)=[a_{n+1},a_{n+2},...]\,.$$

If $y$ belongs to the long interval $f^i(I_n)$ for some $i\in\{0,1,...,q_{n+1}-1\}$,
let $y_n \in I_n$ be such that $f^i(y_n)=y$. Otherwise, $y$ belongs to the short interval
$f^j(I_{n+1})$ for some $j\in\{0,1,...,q_{n}-1\}$, and then let $y_n \in I_{n+1}$ be given by $f^j(y_n)=y$. In the first case, see Figure \ref{trails}, we define
$$\alpha_{n+1}=\frac{\mu\big((x,y_n)\big)}{\mu(I_n)}\,\in[0,1],$$
while in the second case we define
$$\alpha_{n+1}=-\,\frac{\mu\big((y_n,x)\big)}{\mu(I_{n+1})}\,\in[-1,0].$$

\begin{figure}[h!]
\begin{center}~
\hbox to \hsize{\psfrag{0}[][][1]{$0$} \psfrag{In}[][][1]{$I_n$}
\psfrag{In1}[][][1]{$I_{n+1}$}
\psfrag{f}[][][1]{$f^{q_n}$}
\psfrag{F}[][][1]{$f^{q_{n+1}}$}
\psfrag{Y}[][][1]{$Y_n$}
\psfrag{x}[][][1]{$x$}
\psfrag{y}[][][1]{$y_n$}
\psfrag{xn}[][][1]{$f^{q_n}(x)$}
\psfrag{xn1}[][][1]{$f^{q_{n+1}}(x)$}
\psfrag{a}[][][1]{$\alpha_{n+1}=\dfrac{\mu(Y_n)}{\mu(I_n)}$}
\hspace{1.0em} \includegraphics[width=4.5in]{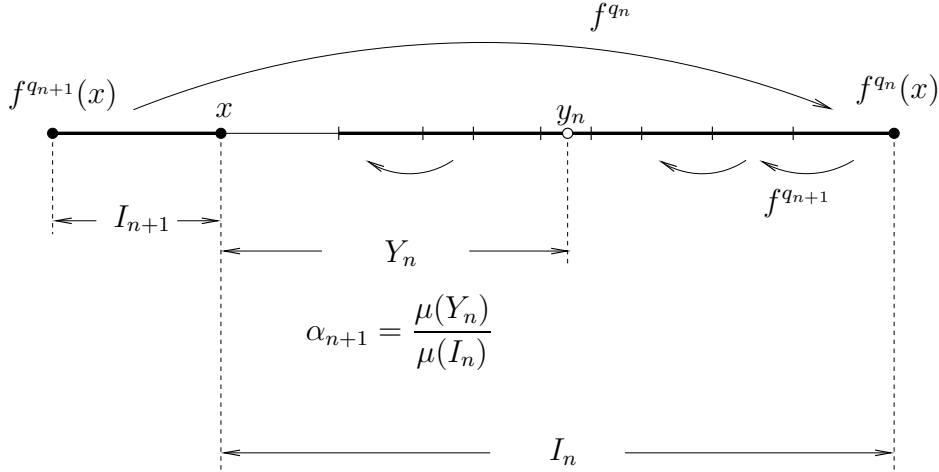}
   }
\end{center}
\caption[trails]{\label{trails} Calculating renormalization trails.}
\end{figure}

The points $y_n$, ${n\geq 0}$, defined above are called the \emph{renormalization ancestors} of $y$ (with respect to $x$ and $f$). But we are more interested in the sequence of pairs $(\rho_n,\alpha_n) \in M=\big([0,1]\setminus\Q\big)\times[-1,1]$. We therefore also give it a name.

\begin{definition}\label{deftrails} The sequence $\big\{(\rho_n,\alpha_n)\big\}_{n \geq 0} \subset M$ is called the \emph{renormalization trail}, or simply the \emph{trail}, of the point $y$ with respect to $x$ and $f$.
\end{definition}

In Section \ref{provaproptrails} we will prove the following result.

\begin{theorem}\label{propresiduais} There exists a full Lebesgue measure set\, $\bm{R}\subset[0,1]$ of irrational numbers with the following property: given a minimal circle homeomorphism $f$ with $\rho(f)\in\bm{R}$ and given any point $x \in S^1$ there exists a set\, $\mathcal{B}_{x}\subset S^1$ which is residual (in the Baire sense) and has full $\mu_f$-measure such that for all $y \in \mathcal{B}_{x}$ the renormalization trail $\big\{(\rho_n,\alpha_n)\big\}$ of $y$ (with respect to $x$ and $f$) is dense in the rectangle $[0,1]\times[-1,1]$.
\end{theorem}

Being dense in $[0,1]$, the orbit under the Gauss map of any element of\, $\bm{R}$ accumulates at the origin. In particular,\, $\bm{R}$ is disjoint from the set $\bt\subset[0,1]$ of \emph{bounded type} numbers. Note also that $\mathcal{B}_{x}$ is disjoint from $\mathcal{O}_f^+(x)=\big\{x,f(x),f^2(x),...\big\}$, since for $n\geq 0$ the second coordinate of the renormalization trail of $f^n(x)$ with respect to $x$ and $f$ eventually becomes constant equal to $0$. As mentioned, the proof of Theorem \ref{propresiduais} will be given in Section \ref{provaproptrails}.

\subsection{Some tools}\label{sectools} We finish Section \ref{sechomeos} reviewing some classical tools from one-dimensional dynamics, that will be used along the text.

One important tool is the control of {\it cross-ratio distortion\/}. There are several cross-ratios used in the
study of one-dimensional dynamical systems, all equivalent. In the present paper (more precisely, in the proof of Lemma \ref{crossratiodistortion}),
we use the following version.
Given two intervals $M\subset T\subset S^{1}$ with $M$ compactly contained in $T$ (written $M\Subset T$) let us denote by $L$ and $R$ the
two connected components of $T\setminus M$. We define the \emph{cross-ratio} of the pair $M,T$ to be the ratio
\begin{equation*}
[M,T]= \frac{|M|\,|T|}{|L|\,|R|} \in (0,\infty).
\end{equation*}

The cross-ratio is preserved by M\"obius transformations. Moreover, it is weakly expanded by maps with negative Schwarzian derivative (see Lemma \ref{contracts} below). To be more precise, let $f:S^{1}\to S^{1}$ be a continuous map, and let $U\subseteq S^{1}$ be an open set such that $f|_{U}$ is a homeomorphism onto its image. If $M\subset T\subset U$ are intervals, with $M\Subset T$,
the \textit{cross-ratio distortion} of the map $f$ on the pair of intervals $(M,T)$ is defined to be the ratio of cross-ratios
\begin{equation*}
\crd(f;M,T)= \frac{\big[f(M),f(T)\big]}{[M,T]}.
\end{equation*}

If $f|T$ is a M\"obius transformation, then we have that $\crd(f;M,T)=1$. When $f|T$ is a diffeomorphism onto its image and $\log{Df}|_{T}$
has {\it bounded variation\/} in $T$ (for instance, if $f$ is a $C^2$ diffeomorphism), we obtain
$\crd(f;M,T)\leq e^{2V}$, where $V=\mathrm{Var}(\log{Df}|T)$.
We shall use the following chain rule in iterated form:
\begin{equation}\label{CRIchainrule}
 \crd(f^j;M,T) = \prod_{i=0}^{j-1} \crd(f;f^{i}(M), f^{i}(T))\ .
\end{equation}

There is a relationship between quasisymmetry and distortion of cross-ratios, but a full discussion of it would constitute a lengthy digression. There is in fact only one place in the present paper (in Section \ref{StheproofofA}) where a particular instance of this relationship is required. What we need is a simple consequence of the following result, which we state without proof (\emph{cf.} \cite[p.~130]{dFdMbook}).

\begin{prop}\label{qscross1} If $\phi:S^1 \to S^1$ is quasisymmetric, then there exists a non-decreasing function $\sigma:[0,\infty)\to [0,\infty)$ with $\sigma(t)\to 0$ as $t\to 0$ such that $[\phi(M),\phi(T)] \leq \sigma([M,T])$ for every pair of intervals $M,T\subset S^1$ with $M$ compactly contained in the interior of $T$.
\end{prop}

A proof of this result may be found in \cite{astala2009}. In order to state the corollary in simple terms, it is best to introduce a definition. We say that a homeomorphism $\phi:S^1 \to S^1$ \emph{has weakly bounded cross-ratio distortion} if for every pair of constants $0<\alpha<\beta<\infty$ there exists $B_{\alpha,\beta}>0$ such that $\crd(\phi,M,T)\leq B_{\alpha,\beta}$ for every pair of intervals $M,T$ (with $M$ compactly contained in the interior of $T$) such that $\alpha\leq [M,T]\leq \beta$. 

\begin{coro}\label{qscross2} Every quasisymmetric homeomorphism of the circle has weakly bounded cross-ratio distortion.
\end{coro}

\begin{proof} This is a straightforward consequence of Proposition \ref{qscross1}. 
\end{proof}

This corollary will be used in its contrapositive, as a criterion for \emph{non-quasi symmetry} (see Section \ref{StheproofofA}).

\medskip

Recall that for a given $C^3$ map $f$, the \textit{Schwarzian derivative} of $f$ is the differential operator defined for all $x$
 regular point of $f$ by:
 \begin{equation*}
  Sf(x)= \dfrac{D^{3}f(x)}{Df(x)} - \dfrac{3}{2} \left( \dfrac{D^{2}f(x)}{Df(x)}\right)^{2}.
 \end{equation*}

The relation between the Schwarzian derivative and cross-ratio distortion is given by the following well known fact.

\begin{lemma}\label{contracts} If $f$ is a $C^3$ diffeomorphism with $Sf<0$, then for any two intervals $M\subset T$ contained in the domain
of $f$ we have $\crd(f;M,T)>1$, that is, $\big[f(M),f(T)\big]>[M,T]$.
\end{lemma}

For a proof of Lemma \ref{contracts} see for instance the appendix in \cite{EdFG}.
We recall now the definition of an \emph{almost parabolic map}, as given in \cite[Section 4.1, page 354]{edsonwelington1}.

\begin{definition}\label{def:apm} An \textit{almost parabolic map} is a negative-Schwarzian $C^3$ diffeomorphism$$\phi \colon  J_1\cup J_2\cup \cdots \cup J_\ell \;\to\; J_2\cup J_3\cup \cdots \cup J_{\ell+1},$$such that $\phi(J_k)= J_{k+1}$ for all $1\leq k\leq \ell$, where $J_1,J_2, \ldots, J_{\ell+1}$ are consecutive intervals on the circle (or on the line). The positive integer $\ell$ is called the \textit{length} of $\phi$, and the positive real number
  \[
    \sigma =\min\left\{\frac{|J_1|}{|\cup_{k=1}^\ell J_k|}\,,\, \frac{|J_\ell|}{|\cup_{k=1}^\ell J_k|}     \right\}
  \]is called the \textit{width\/} of $\phi$.
  \end{definition}

The fundamental geometric control on almost parabolic maps is given by the following result.

  \begin{lemma}[Yoccoz's lemma]\label{yoccozlemma}
  Let $\phi \colon \bigcup_{k=1}^\ell J_k \to \bigcup_{k=2}^{\ell+1} J_k$ be an almost parabolic map with length $\ell$ and
  width $\sigma$. There exists a constant $C_\sigma>1$ (depending on $\sigma$ but not on $\ell$) such that,
  for all $k=1,2,\ldots,\ell$, we have
  \begin{equation}\label{yocineq}
    \frac{C_\sigma^{-1}|I|}{[\min\{k,\ell+1-k\}]^2} \;\leq\; |J_k| \;\leq\;  \frac{C_\sigma|I|}{[\min\{k,\ell+1-k\}]^2}\ ,
  \end{equation}
  where $I=\bigcup_{k=1}^\ell J_k$ is the domain of $\phi$.
  \end{lemma}

For a proof of Lemma \ref{yoccozlemma} see \cite[Appendix B, page 386]{edsonwelington1}. To be allowed to use Yoccoz's lemma we will
need the following result.

\begin{lemma}\label{negschwarz} For any given multicritical circle map $f$ there exists $n_0=n_0(f)\in\nt$ such that for all $n \geq n_0$ we have
that
\[ Sf^{j}(x)<0\quad\text{for all $j\in \{1, \cdots, q_{n+1}\}$ and for all $x \in I_{n}$ regular point of $f^{j}$.}
 \]
Likewise, we have
\[ Sf^{j}(x)<0\quad\text{for all $j\in \{1, \cdots, q_{n}\}$ and for all $x \in I_{n+1}$ regular point of $f^j$}.
 \]
\end{lemma}

For a proof of Lemma \ref{negschwarz} see \cite[Lemma 4.1, page 852]{EdFG}. The following lemma is an adaptation of \cite[Lemma 4.2, page 5600]{EdF}. Let $x\in S^1$ and consider the associated dynamical partitions
$\mathcal{P}_n(x)$.

\begin{lemma}\label{criticalspots}
 Let $0\leq k<a_{n+1}$ be such that the interval $\Delta_{k,n}=f^{q_n+kq_{n+1}}(I_{n+1}(x)) \subset I_n(x)$ contains a critical
 point of $f^{q_{n+1}}$. Then{\footnote{Given positive numbers $a$ and $b$, we write $a\asymp b$ to mean that there exists a constant $C>1$, which is either absolute or depends on the real bounds for the map $f$, such that $C^{-1}b\leq a\leq Cb$.}} $|\Delta_{k,n}|\asymp |I_n(x)|$.
\end{lemma}

In the statement given in  \cite[Lemma 4.2, page 5600]{EdF}, it is assumed that $x$ is a critical point of $f$, but the
proof given there also works when $x$ is not critical.
An interval such as $\Delta_{k,n}$ appearing in the statement above, containing some critical point of $f^{q_{n+1}}$, is
called a {\it critical spot\/} (at level $n$). Thus, Lemma \ref{criticalspots} is saying that every critical spot is
large, {\it i.e.,} is comparable to the atom of $\mathcal{P}_n(x)$ in which it is contained.

\section{The skew product}\label{Secskew}

In this section we construct a \emph{skew product} (see \S\ref{secfirstskew} below) that will be crucial in order to prove Theorem \ref{propresiduais} (its proof will be given in \S \ref{provaproptrails}) and also to prove Theorem \ref{ThmB} (see Section \ref{StheproofofA}).

\subsection{The fiber maps}\label{SecFibermaps} For any given $\rho\in[0,1]\setminus\Q$ consider the piecewise-affine dynamical system $T_{\rho}:[-1,1]\to[-1,1]$ given by:
\[T_{\rho}(\alpha)=
\begin{dcases}
-\alpha & \mbox{for $\alpha\in\big[-1,0\big]$}\\[1ex]
-\,\frac{\alpha}{\rho\,G(\rho)} & \mbox{for $\alpha\in\big[0,\rho\,G(\rho)\big]$}\\[1ex]
\left\{\frac{1-\alpha}{\rho}\right\} & \mbox{for $\alpha\in\big(\rho\,G(\rho),1\big]$,}\\
\end{dcases}
\]where $G$ is the Gauss map introduced in \S\ref{gaussmap}. Each $T_{\rho}$ is a \emph{Markov} map, its graph is depicted in Figure \ref{markov}. 

\begin{figure}[h!]
\begin{center}~
\hbox to \hsize{\psfrag{0}[][][1]{$0$} \psfrag{1}[][][1]{$\rho$}
\psfrag{1}[][][1]{$1$}
\psfrag{-1}[][][1]{$\!\!\!-1$}
\psfrag{0}[][][1]{$0$}
\psfrag{p}[][][1]{$\hat{\rho}_0$}
\psfrag{2p}[][][1]{$\hat{\rho}_1$}
\psfrag{cd}[][][1]{$\cdots$}
\psfrag{P}[][][1]{$\hat{\rho}_{a_0-1}$}
\psfrag{1p}[][][1]{$1\!\!=\!\!\hat{\rho}_{a_0}$}
\hspace{1.0em} \includegraphics[width=4.5in]{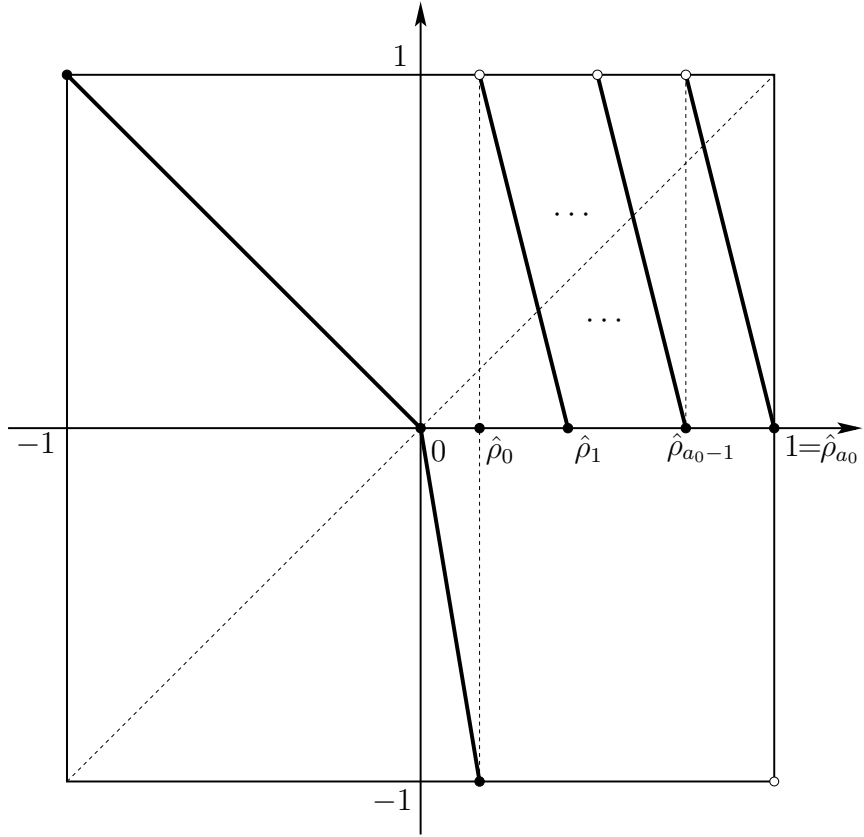}
   }
\end{center}
\caption[markov]{\label{markov} The fiber map $T_\rho$\,; here, $\hat{\rho}_j=(G(\rho)+j)\rho$ for each $0\leq j\leq a_0$, where $a_0=\lfloor\frac{1}{\rho}\rfloor$.}
\end{figure}

\subsection{The skew product}\label{secfirstskew} As before (see \S \ref{sectrails}) we consider the rectangle $R=[0,1]\times[-1,1]$ in $\R^2$, and let $M=\big([0,1]\setminus\Q\big)\times[-1,1] \subset R$. Consider the skew product $T:M \to M$ given by:$$T(\rho,\alpha)=\big(G(\rho)\,,\,T_{\rho}(\alpha)\big)\,,$$where $G$ is the Gauss map introduced in \S \ref{gaussmap}, and where the fiber maps $T_{\rho}$ were introduced in the previous section (\S \ref{SecFibermaps}). The main dynamical property of the skew product $T$ that we will need in this paper is the following.

\begin{prop}\label{genorbdensas} There exists a set $\mathcal{G}_0 \subset [0,1]\times[-1,1]$, which is residual (in the Baire sense) and has full Lebesgue measure, such that any initial condition in $\mathcal{G}_0$ has a positive orbit under $T$ which is dense in\, $[0,1]\times[-1,1]$.
\end{prop}

The set $\mathcal{G}_0$ given by Proposition \ref{genorbdensas} will be crucial in the proof of Theorem \ref{propresiduais} (which will be given in Section \ref{provaproptrails} below), and also in the proof of Theorem \ref{ThmB} (see Section \ref{StheproofofA}). In Section \ref{S:real} we will also need the following fact.

\begin{lemma}[Topologically Exactness]\label{Tcov} Let $U$ be a subset of the rectangle $R$ with non-empty interior. Then there exists $n\in\nt$ such that $T^n(U \cap M)=M$.
\end{lemma}

We postpone the proofs of Proposition \ref{genorbdensas} and Lemma \ref{Tcov} until Appendix \ref{apperg}.

\section{Proof of Theorem \ref{propresiduais}}\label{provaproptrails}

Just as in Section \ref{sechomeos}, let $f:S^1 \to S^1$ be an orientation preserving circle homeomorphism with irrational rotation number $\rho$. With Proposition \ref{genorbdensas} at hand, Theorem \ref{propresiduais} will be a straightforward consequence of the following fact:

\begin{lemma}\label{leminhacomb} Given $x$ and $y$ in $S^1$ we have:$$(\rho_{n},\alpha_{n})=T^n(\rho_0,\alpha_0)\quad\mbox{for all $n\in\nt$,}$$where $\{(\rho_n,\alpha_n)\}$ is the renormalization trail of $y$ with respect to $x$ and $f$, as defined in Section \ref{sectrails}, and $T:M \to M$ is the skew product constructed in Section \ref{secfirstskew}.
\end{lemma}

During the proof of Lemma \ref{leminhacomb}, we will make repeated use of the well-known formula\, $\mu(I_n)=\prod_{j=0}^{j=n}G^j(\rho)$ (see Lemma \ref{formmuIn} in Section \ref{gaussmap}).

\begin{proof}[Proof of Lemma \ref{leminhacomb}] By our definition of renormalization trails, $\rho_n=G^n(\rho)$ for all $n\in\nt$, which coincides with the definition of the skew product $T$. So we only need to deal with the second coordinate of the trails. Let us treat first the cases $n=0$ and $n=1$: if $y$ belongs to the short interval $I_1=\big(f^{a_0}(x),x\big]$ (see Section \ref{subseccombpart}), we have $\alpha_0\in\big[0,\rho_0\,G(\rho_0)\big]$ and then:$$T_{\rho_0}(\alpha_0)=T_{\rho_0}\big(\mu((x,y))\big)=-\,\frac{\mu((x,y))}{\rho_0\,G(\rho_0)}=-\,\frac{\mu\big((x,y)\big)}{\mu(I_{1})}=\alpha_1\,.
$$

Otherwise, there exist $y_0$ in the long interval $I_0=\big(x,f(x)\big]$ and $i\in\{0,1,...,a_0-1\}$ such that $f^i(y_0)=y$, in which case we have $\alpha_0\in\big[\rho_0\,G(\rho_0),1\big]$ and then:
\begin{align*}
T_{\rho_0}(\alpha_0)&=T_{\rho_0}\big(1-\mu((x,y_0))-i\,\rho_0\big)=\left\{\frac{\mu((x,y_0))+i\,\rho_0}{\rho_0}\right\}=\left\{\frac{\mu((x,y_0))}{\rho_0}\right\}=\\
&=\frac{\mu((x,y_0))}{\rho_0}=\frac{\mu((x,y_0))}{\mu(I_0)}=\alpha_1\,.
\end{align*}

In any case, $\alpha_1=T_{\rho_0}(\alpha_0)$ and then $(\rho_1,\alpha_1)=T(\rho_0,\alpha_0)$, as desired. Therefore, in order to prove Lemma \ref{leminhacomb} we have, for each $n\in\nt$, three possible cases to consider:

\begin{enumerate}

\item If $y_n \in I_{n+2}$\,, we have:
$$0\leq\alpha_{n+1}=\frac{\mu((x,y_n))}{\mu(I_n)}\leq\frac{\mu(I_{n+2})}{\mu(I_n)}=\rho_{n+1}\,G(\rho_{n+1})\,,$$and then:$$T_{\rho_{n+1}}(\alpha_{n+1})=-\,\frac{\alpha_{n+1}}{\rho_{n+1}\,G(\rho_{n+1})}=-\,\frac{\alpha_{n+1}\,\mu(I_n)}{\mu(I_{n+2})}=-\,\frac{\mu((x,y_n))}{\mu(I_{n+2})}=\alpha_{n+2}\,.$$

\item If $y_n \in I_n \setminus I_{n+2}$\,, we have:$$\frac{\mu(I_{n+2})}{\mu(I_{n})}<\alpha_{n+1}\leq 1\,,$$which implies $\alpha_{n+1}\in\big(\rho_{n+1}\,G(\rho_{n+1}),1\big]$, and then $\displaystyle T_{\rho_{n+1}}(\alpha_{n+1})=\left\{\frac{1-\alpha_{n+1}}{\rho_{n+1}}\right\}$.
Consider the fundamental domains $\Delta_{j,n}\subset I_n$ for $f^{q_{n+1}}$ given by$$\Delta_{j,n}=f^{j\,q_{n+1}+q_n}(I_{n+1})=
\big(f^{(j+1)\,q_{n+1}+q_n}(x),f^{j\,q_{n+1}+q_n}(x)\big]$$for $j\in\{0,1,...,a_{n+1}-1\}$, and let $\ell_n\in\{0,1,...,a_{n+1}-1\}$ be defined by $y_n\in \Delta_{\ell_n,n}$.
We claim that $\displaystyle\ell_n=\left\lfloor\frac{1-\alpha_{n+1}}{\rho_{n+1}}\right\rfloor$. Indeed, since $\mu\big(\Delta_{j,n}\big)=\mu(I_{n+1})$ for all $j\in\{0,1,...,a_{n+1}-1\}$, we get that:
$$\ell_n\,\mu(I_{n+1})\leq(1-\alpha_{n+1})\,\mu(I_{n})<(\ell_n+1)\,\mu(I_{n+1}).$$
Equivalently
$$\ell_n\leq(1-\alpha_{n+1})\,\frac{\mu(I_n)}{\mu(I_{n+1})}<\ell_n+1\,.$$
Finally, from$$\frac{\mu(I_n)}{\mu(I_{n+1})}=\frac{\prod_{j=0}^{j=n}G^j(\rho)}{\prod_{j=0}^{n+1}G^j(\rho)}=\frac{1}{G^{n+1}(\rho)}=\frac{1}{\rho_{n+1}}\,,$$we deduce that $\displaystyle\ell_n\leq\frac{1-\alpha_{n+1}}{\rho_{n+1}}<\ell_n+1$\,, which implies the claim. With this at hand we obtain:
\begin{align*}
T_{\rho_{n+1}}(\alpha_{n+1})&=\left\{\frac{1-\alpha_{n+1}}{\rho_{n+1}}\right\}=\frac{1-\alpha_{n+1}}{\rho_{n+1}}-\,\ell_n=\frac{\mu(I_n)-\alpha_{n+1}\,\mu(I_n)}{\mu(I_{n+1})}-\,\ell_n\\
&=\frac{\mu(I_n)-\big[\mu((x,y_n))+\ell_n\,\mu(I_{n+1})\big]}{\mu(I_{n+1})}=\alpha_{n+2}\,.
\end{align*}

\item Whenever $y_n$ belongs to the \emph{short} interval $I_{n+1}$, we have $\alpha_{n+1}\in[-1,0)$ and then $T_{\rho_{n+1}}(\alpha_{n+1})=-\alpha_{n+1}=\alpha_{n+2}$\,, since $y_{n+1}=y_n$ belongs now to the \emph{long} interval $I_{n+1}$.
\end{enumerate}
\end{proof}

\begin{proof}[Proof of Theorem \ref{propresiduais}] Let $\mathcal{G}_0 \subset R$ be given by Proposition \ref{genorbdensas}. By Fubini's theorem, there exists a full Lebesgue measure set $\bm{R}\subset[0,1]$ such that for each $\rho\in\bm{R}$, the set $\bm{R}_{\rho}=\big\{\alpha\in[-1,1]:(\rho,\alpha)\in\mathcal{G}_0\big\}$ has full Lebesgue measure in $[-1,1]$. In particular, $\bm{R}_{\rho}$ is also residual\footnote{Indeed, let $\{A_n\}$ be a sequence of open and dense sets in $R$ such that $\cap A_n=\mathcal{G}_0$. For each $\rho\in\bm{R}$ and each $n$ we have that $\big(\{\rho\}\times[-1,1]\big)\cap A_n$ is open and has full Lebesgue measure in $\{\rho\}\times[-1,1]$, and in particular it is also dense in $\{\rho\}\times[-1,1]$.} in $[-1,1]$ for all $\rho\in\bm{R}$. Given a minimal circle homeomorphism $f$ with $\rho(f)\in\bm{R}$ and given any point $x \in S^1$, the map that sends $\alpha\in(0,1)$ to the point $y \in S^1\setminus\{x\}$ which satisfies $\mu_f\big([x,y]\big)=\alpha$ (and note that such point is unique if we fix, say, the counterclockwise orientation) is a homeomorphism that, by definition, identifies the Lebesgue measure in $(0,1)$ with the probability measure $\mu_{f}$ in $S^1\setminus\{x\}$. By combining Proposition \ref{genorbdensas} with Lemma \ref{leminhacomb}, we deduce that it is enough to take $\mathcal{B}_{x}$ as the image (under the homeomorphism described above) of $\bm{R}_{\rho}\cap(0,1)$.
\end{proof}

\section{Even-type rotation numbers} \label{sec:eventype}

Let us now present a result concerning trails for maps whose rotation number belongs to the special class appearing in the statements of Theorem \ref{ThmB} and Theorem \ref{ThmC}. We denote by $\mathbb{E}$ the set of those irrationals $0<\theta<1$ all of whose partial quotients $a_n(\theta)$ are {\it even\/} (in particular $a_n(\theta)\geq 2$ for all $n$). We also consider the subset $\mathbb{E}_\infty=\{\theta\in \mathbb{E}\,:\, \lim_{n\to\infty} a_n(\theta)=\infty\}$.

\begin{remark} We note {\it en-passant\/} that $\mathbb{E}_\infty$ contains some Diophantine numbers: for example, the number $\theta=[a_1,a_2,\ldots,a_n,\ldots]$ with $a_n=2^n$ is Diophantine, and it clearly belongs to $\mathbb{E}_\infty$. The set $\mathbb{E}_{\infty}$ also contains many Liouville numbers: for instance, any $\theta=[a_1,a_2,\ldots,a_n,\ldots]$ with $a_n$ even and $a_n > e^{n^n}$ for all $n\in\nt$ belongs to $\mathbb{E}_\infty$. Finally, note that the transcendental number $\lambda=(e-1)/(e+1)$ also belongs to $\mathbb{E}_\infty$; indeed, its continued fraction expansion has $a_n=4n-2$ for all $n\geq 1$, {\it i.e.,}
$\lambda=[2,6,10,14,\ldots]$ -- this is a special case of an old identity due to Euler and Lambert{\footnote{Which states that $\tanh{(x^{-1})}=[x,3x,5x,7x,\ldots]$ for all $x\in \mathbb{N}$; see \cite[p.~71]{Lang}}}.
\end{remark}

\begin{prop}\label{evennumbers} Let $f: S^1\to S^1$ be a minimal circle homeomorphism with $\rho(f)=\rho$. Given $x,y\in S^1$ distinct, let $\{(\rho_n,\alpha_n)\}_{n\geq 0}$
 be the renormalization trail of $y$ with respect to $x$ and $f$. If $\rho\in \mathbb{E}$ and $\alpha_0=\frac{1}{2}$, then for all $n\geq 1$ we have
 $\rho_n<\frac{1}{2}$, and
 \begin{equation}\label{evenodd}
  \alpha_n \;=\; \left\{\begin{array}{ll}
          {\displaystyle{\frac{\rho_n}{2}}} & \mbox{if $n$ is odd,} \\
          {}&{}\\
          {\displaystyle{\frac{1}{2} + \rho_n}} & \mbox{if $n$ is even.}
         \end{array}\right.
 \end{equation}
 In particular, if $\rho\in \mathbb{E}_\infty$, then there exists a subsequence $n_i\to\infty$ such that $\alpha_{n_i}\to \frac{1}{2}$.
\end{prop}

\begin{proof}[Proof of Proposition \ref{evennumbers}] First note that, if $a_0,a_1,a_2,\ldots$ are the partial quotients of the continued fraction expansion of $\rho_0$,
 then by hypothesis $a_n\geq 2$ for all $n$, and this already implies that $\rho_n<\frac{1}{a_{n}}\leq\frac{1}{2}$ for all $n\geq 1$.
 This takes care of the first assertion in the statement. In order to prove the second assertion, we will use Lemma \ref{leminhacomb}
 and induction on $n$.

{\noindent\it (1) Base of induction.\/} We have $\alpha_0=\frac{1}{2}$, and since $\alpha_0>\rho_0G(\rho_0)=\rho_0\rho_1$, Lemma \ref{leminhacomb} tells us that
\[
\alpha_1=T_{\rho_0}(\alpha_0) = \left\{\frac{1-\alpha_0}{\rho_0}\right\}
= \left\{\frac{1}{2\rho_0}\right\}\ .
\]
But $\rho_0^{-1}=a_0+\rho_1$, where $a_0\geq 2$ is {\it even\/}. Therefore
\[
 \alpha_1\;=\; \left\{\frac{1}{2}(a_0+\rho_1)\right\}\;=\;  
 \frac{\rho_1}{2}\ .
\]
This verifies \eqref{evenodd} for $n=1$. 
Let us now look at $\alpha_2$. 
We have $\alpha_1>\rho_1G(\rho_1)=\rho_1\rho_2$. Hence, using Lemma \ref{leminhacomb} and the fact that 
$\rho_1^{-1}=a_1+\rho_2$, we see that
\begin{align*}
 \alpha_2\;=\; T_{\rho_1}(\alpha_1)\;&=\; \left\{\frac{1-\alpha_1}{\rho_1}\right\} \\
  &=\; \left\{\frac{1}{\rho_1}-\frac{1}{2}\right\}\\
  &=\; \left\{a_1+\rho_2-\frac{1}{2}\right\}\\
  &=\; \left\{\rho_2-\frac{1}{2}\right\}\\
  &=\; \frac{1}{2} + \rho_2\ .
\end{align*}
This verifies \eqref{evenodd} for $n=2$. Summarizing, we have established the base of the induction.
\medskip

{\noindent\it (2) Induction step.\/} Suppose \eqref{evenodd} holds true for $n$.
  In order to show that this assertion holds true for $n+1$, there are two
  cases to consider, according to whether $n$ is odd or even.
  \begin{enumerate}
  \item[(i)] If $n$ is odd, then we are assuming that $\alpha_n=\frac{1}{2}\rho_n$.
  In particular, we have $\alpha_n>\rho_n\rho_{n+1}=\rho_nG(\rho_n)$, so 
  Lemma \ref{leminhacomb} tells us that
  \begin{align*}
   \alpha_{n+1}\;=\;T_{\rho_n}(\alpha_n)\;&=\; \left\{\frac{1-\alpha_n}{\rho_n}\right\} \\
   &=\; \left\{\frac{1}{\rho_n}- \frac{1}{2}\right\} 
  \end{align*}
  Using here that $\rho_n^{-1}=a_{n}+\rho_{n+1}$, we get
  \[
   \alpha_{n+1}=\left\{a_{n}+\rho_{n+1}-\frac{1}{2}\right\} = \frac{1}{2}+\rho_{n+1}\ .
  \]
  This establishes the induction step when $n$ is odd.

  \item[(ii)] If $n$ is even, then we are assuming that $\alpha_n = \frac{1}{2}+ \rho_n$, by the induction hypothesis.
  Hence we have $\alpha_n >\frac{1}{2} > \rho_n\rho_{n+1} = \rho_nG(\rho_n)$, and therefore from Lemma \ref{leminhacomb} we deduce that
  \begin{align}\label{oddalpha}
   \alpha_{n+1}=T_{\rho_n}(\alpha_n) \;&=\; \left\{\frac{1-\alpha_n}{\rho_n}\right\} \nonumber \\
   &=\; \left\{\frac{1}{2\rho_n}- 1\right\} \nonumber\\
   &=\; \left\{\frac{1}{2\rho_n}\right\} \ .
  \end{align}
  Again, using that $\rho_n^{-1}=a_{n}+\rho_{n+1}$, we 
  see that
  \[
   \alpha_{n+1}\;=\; \left\{\frac{1}{2}a_{n+1}+\frac{1}{2}\rho_{n+1}\right\}\;=\; \frac{\rho_{n+1}}{2}\ ,
  \]
  where in the last equality we have at last used the fact that $a_{n}$ is an even integer!
  This establishes the induction step when $n$ is even, and completes the proof of the second assertion.
 \end{enumerate}
 Finally, the last assertion in the statement is easily proved: if $\rho\in \mathbb{E}_\infty$, then $\rho_n\to 0$ as $n\to \infty$.
 Hence by \eqref{evenodd} we see that $\alpha_{2i}\to \frac{1}{2}$ as $i\to\infty$. This concludes the proof.
\end{proof}

\begin{remark} The above proof still works if only the odd partial quotients $a_{2k+1}$ are required to be even (but still requiring $a_n \neq 1$ for all $n$). The resulting class of numbers with this property is a bit larger than $\mathbb{E}$, but still has zero Lebesgue measure.
\end{remark}

\section{Proofs of Theorems \ref{ThmA} and \ref{ThmB}}\label{StheproofofA}

In this section we prove our first two main results, namely Theorem \ref{ThmA} and Theorem \ref{ThmB}. We first recall the setup for both theorems, and fix some notation.

Let $f,g:S^1\to S^1$ be  two $C^3$ (multi)critical circle maps with the same irrational rotation number
$\rho=[a_0,a_1,\ldots,a_n,\ldots]$. Let $h:S^1\to S^1$ be a
topological conjugacy between $f$ and $g$ mapping orbits of $f$ to orbits of $g$ ({\it i.e.,}, such that $h\circ f=g\circ h$).
Let $x, z\in S^1$ be such that $h(x)=z$. Suppose also that $w\in S^1$, $w\neq z$, is a critical point for $g$.
Assume one of the following two scenarios (which correspond to the situations in Theorems \ref{ThmA} and \ref{ThmB}, respectively).

\begin{itemize}
 \item[] {\sl{Scenario A\/}.} Both $f$ and $g$ are {\it uni-critical\/} circle maps, with critical points at $x$ and $w$, respectively.
 \item[] {\sl{Scenario B\/}.} The map $f$ is {\it uni-critical\/} with critical point at $x$, whereas the map $g$ is {\it bi-critical\/} with critical
 points at $z$ and $w$.
\end{itemize}

In either scenario, let $y=h^{-1}(w)$ and let $y_n$, $n\geq 0$, be the renormalization ancestors of $y$ (with respect to $x$ and $f$).
Likewise, let $w_n=h(y_n)$, $n\geq 0$, denote the renormalization ancestors of $w=h(y)$ (with respect to $z$ and $g$).
Finally, let $(\rho_n,\alpha_n)$, $n\geq 0$, be the renormalization trail of $y$ (with respect to $x$ and $f$)
-- which is also the renormalization trail of $w$ (with respect to $z$ and $g$).

Both Theorem \ref{ThmA} and Theorem \ref{ThmB} will be straightforward consequences of the following result.

\begin{figure}[t]
\begin{center}~
\hbox to \hsize{\psfrag{0}[][][1]{$0$} \psfrag{x}[][][1]{$x$}
\psfrag{z}[][][1]{$z$}
\psfrag{d}[][][1]{$\Delta_{n_i}$}
\psfrag{L}[][][1]{$L_{n_i}$}
\psfrag{R}[][][1]{$R_{n_i}$}
\psfrag{hd}[][][1]{$h(\!\Delta_{n_i}\!)$}
\psfrag{hL}[][][1]{$h(L_{n_i})$}
\psfrag{hR}[][][1]{$h(R_{n_i})$}
\psfrag{f0}[][][1]{$f^{q_{n_i}}(x)$}
\psfrag{f1}[][][1]{$f^{q_{n_i+1}}(x)$}
\psfrag{g0}[][][1]{$g^{q_{n_i}}(z)$}
\psfrag{g1}[][][1]{$g^{q_{n_i+1}}(z)$}
\psfrag{y}[][][1]{$y_{n_i}$}
\psfrag{w}[][][1]{$\!\!w_{n_i}$}
\hspace{1.0em} \includegraphics[width=5.2in]{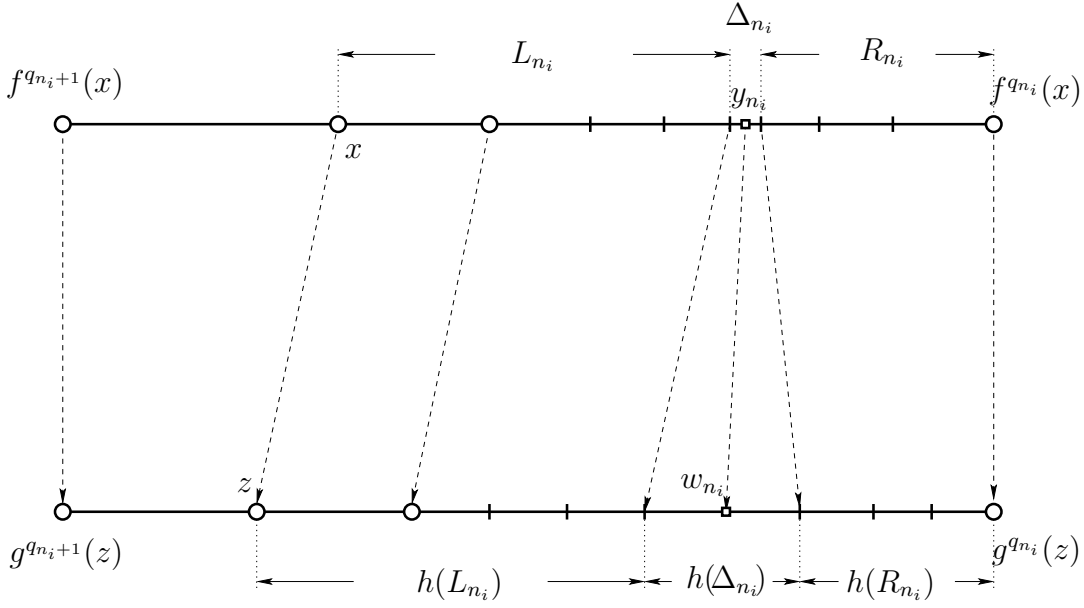}
   }
\end{center}
  \caption[markov]{\label{cross} The distortion of cross-ratios is large.}
\end{figure}

\begin{lemma}\label{crossratiodistortion}
 In either of the two scenarios above, suppose that there exists a subsequence $n_i\to\infty$ such
 that $\rho_{n_i+1}\to 0$ as $i\to\infty$, and
 $\left|\alpha_{n_i+1}-\frac{1}{2}\right|<\frac{1}{4}$ for all $i$. Then the conjugacy $h$ is \emph{not} quasisymmetric.
\end{lemma}

\begin{proof}[Proof of Lemma \ref{crossratiodistortion}] The idea is to show that $h$ has unbounded distortion of cross-ratios (see \S \ref{sectools} for the definition, and recall from Corollary \ref{qscross2} that this implies that $h$ is not quasisymmetric). Passing to a subsequence if necessary, we may assume that either (a) $y_{n_i}\in I_{n_i}$ for all $i$; or (b) $y_{n_i}\in I_{n_i+1}$ for all $i$. We give the proof assuming that case (a) holds. The proof in case (b) is the same, {\it mutatis mutandis\/}.

By restricting our attention to sufficiently large $i$, we may assume that $\rho_{n_i+1}<\frac{1}{9}$, which implies that $a_{n_i+1} > 8$. Then we must have $y_{n_i}\in I_{n_i}\setminus I_{n_i+2}$. Indeed, if $y_{n_i}\in I_{n_i+2}$, then $\alpha_{n_i+1}\leq \mu(I_{n_i+2})/\mu(I_{n_i})<\frac{1}{9}$, which contradicts the hypothesis. Since the intervals
\begin{equation}\label{ydelta}
\Delta^{(j)}\;=\;f^{q_{n_i}+jq_{n_i+1}}(I_{n_i+1})\ ,\ \ \ \  0\leq j\leq a_{n_i+1}-1\ ,
\end{equation}
constitute a partition of $I_{n_i}\setminus I_{n_i+2}$ (modulo endpoints), it follows that there exists $0\leq k_{n_i}\leq a_{n_i+1}-1$ such that
 $y_{n_i}\;\in\; \Delta_{n_i}=\Delta^{(k_{n_i})}$.
 \medskip

 {\it Claim.\/} We have $k_{n_i} \asymp a_{n_i+1}\asymp a_{n_i+1}-k_{n_i}  $.
 \medskip

 In order to prove this claim, we first recall that
 \begin{equation}\label{alphamu}
  1-\alpha_{n_i+1}\;=\; \frac{\mu([y_{n_i},f^{q_{n_i}}(x)])}{\mu(I_{n_i})}\ ,
 \end{equation}
 where as before $\mu$ is the unique Borel probability measure invariant under $f$. Moreover, we have
 \begin{equation}\label{squeezeint}
  \bigcup_{j=0}^{k_{n_i}-1} \Delta^{(j)}\;\subseteq \;  [y_{n_i},f^{q_{n_i}}(x)]  \;\subseteq \;
  \Delta_{n_i} \;\cup\; \bigcup_{j=0}^{k_{n_i}-1} \Delta^{(j)}\ .
 \end{equation}
 Since $\mu(\Delta^{(j)}) = \mu(I_{n_i+1})$ for all $j$, from \eqref{alphamu} and \eqref{squeezeint} we get
 \begin{equation}\label{squeezealpha}
  k_{n_i}\,\frac{\mu(I_{n_i+1})}{\mu(I_{n_i})}\;\leq \;1- \alpha_{n_i+1}\;\leq \; (k_{n_i}+1)\,\frac{\mu(I_{n_i+1})}{\mu(I_{n_i})}\ .
 \end{equation}
 Taking into account that
 \[
  \rho_{n_i+1}\;=\; \frac{\mu(I_{n_i+1})}{\mu(I_{n_i})}
 \]
 and that, by hypothesis, $\frac{1}{4}<1-\alpha_{n_i+1}<\frac{3}{4}$, we deduce from \eqref{squeezealpha} that
 \[
  \frac{1}{4\rho_{n_i+1}} - 1 \;<\;k_{n_i}\;<\; \frac{3}{4\rho_{n_i+1}}\ .
 \]
 But $\rho_{n_i+1}^{-1} = a_{n_i+1} +\rho_{n_i+2}$, and $0<\rho_{n_i+2}<1$, so
 \[
  \frac{1}{4}-\frac{1}{a_{n_i+1}}\;<\; \frac{k_{n_i}}{a_{n_i+1}}\;<\; \frac{3}{4}\left(1+\frac{1}{a_{n_i+1}}\right)\ ,
 \]
 and since $\rho_{n_i+1}<\frac{1}{9}$ implies $a_{n_i+1}>8$, we deduce that
 \[
  \frac{1}{8}\;<\;\frac{k_{n_i}}{a_{n_i+1}}\;<\; \frac{27}{32}\ .
 \]
 This proves the claim.

Now, provided $n_i$ is sufficiently large, the map $f^{q_{n_i+1}}$ restricted to the interval $I_{n_i}\setminus I_{n_i+2}$ is an almost parabolic map (see Definition \ref{def:apm} in Section \ref{sectools}). Here we need $n_i$ large enough so that, restricted to the interval in question, the map $f^{q_{n_i+1}}$ is a diffeomorphism with negative Schwarzian derivative, and this is true by Lemma \ref{negschwarz}. By Yoccoz's Lemma \ref{yoccozlemma} and the above claim, we have
 \[
  \frac{|\Delta_{k_{n_i}|}}{|I_{n_i}|}\;\asymp\; \frac{1}{\min\{k_{n_i}^2\,,\, (a_{n_i+1}-k_{n_i})^2\}}\;\asymp\; \frac{1}{a_{n_i+1}^2}\ .
 \]

Letting $L_{n_i}$ and $R_{n_i}$ denote the left and right components of $I_{n_i}\setminus \Delta_{n_i}$, we know from the real bounds (Theorem \ref{realbounds}) that $|L_{n_i}|\asymp |I_{n_i}|\asymp |R_{n_i}|$.
 Therefore we see that
 \begin{equation}\label{firstcross}
  [\Delta_{n_i},I_{n_i}] \;=\; \frac{|\Delta_{n_i}||I_{n_i}|}{|L_{n_i}||R_{n_i}|}\;\asymp\; \frac{1}{a_{n_i+1}^2}\ .
 \end{equation}
 The next step is to estimate the cross-ratio determined by the pair of intervals $h(\Delta_{n_i})$ and $h(I_{n_i})$.
 Here, we first note that $w_{n_i}=h(y_{n_i})\in h(\Delta_{n_i})$ is a critical point for the map $g^{q_{n_i+1}}$; in the terminology
 of \cite{EdF}, $h(\Delta_{n_i})$ is therefore a {\it critical spot\/} of $g^{q_{n_i+1}}|_{h(I_{n_i})}$. 
 As we saw in Lemma \ref{criticalspots},
 every critical spot of a renormalization return map is comparable to the interval domain of said return map. Hence we have
 $|h(\Delta_{n_i})|\asymp |h(I_{n_i})|$. Moreover, by the real bounds for $g$, we have $|h(L_{n_i})|\asymp |h(I_{n_i})|\asymp |h(R_{n_i})|$.
 These facts show that
 \begin{equation}\label{secondcross}
  [h(\Delta_{n_i}),h(I_{n_i})] \;=\; \frac{|h(\Delta_{n_i})||h(I_{n_i})|}{|h(L_{n_i})||h(R_{n_i})|}\;\asymp\; 1\ .
 \end{equation}
 Combining \eqref{firstcross} with \eqref{secondcross}, we finally get an estimate on the cross-ratio distortion of
 the pair of intervals $\Delta_{n_i}\subset I_{n_i}$ under $h$, to wit
 \[
  \crd(h;\Delta_{n_i}\,,\,I_{n_i}) \;=\; \frac{[h(\Delta_{n_i}),h(I_{n_i})]}{[\Delta_{n_i},I_{n_i}]}
  \;\asymp\; a_{n_i+1}^2\ .
 \]But since $\rho_{n_i+1}\to 0$, we have $a_{n_i+1}\to \infty$. This shows that the cross-ratio distortion of $h$ blows up, and so $h$ cannot be quasisymmetric (recall Corollary \ref{qscross2}). The proof of Lemma \ref{crossratiodistortion} is complete.
\end{proof}

\begin{proof}[Proof of Theorem \ref{ThmA}] Consider the sets $\bm{R}$ and $\mathcal{B}_{c_f}$ given by Theorem \ref{propresiduais} (applied to $f$ and $x=c_f$), and define $\bm{R}_A=\bm{R}$. Then Lemma \ref{crossratiodistortion} (applied in the {\sl{Scenario A\/}} case) implies that $\mathcal{B}_{c_f}\subset\mathcal{B}$, which proves Theorem \ref{ThmA}. Remember also that, as explained in Section \ref{secenunciadoThmA}, the fact that the complement of $\mathcal{B}$ is dense follows from the fact that it is non-empty and invariant under the minimal homeomorphism $f$.
\end{proof}

\begin{proof}[Proof of Theorem \ref{ThmB}] By Lemma \ref{crossratiodistortion} (applied in the {\sl{Scenario B\/}} case), it is enough to consider$$\mathcal{G}=\mathcal{G}_0\cup\left(\mathbb{E}_{\infty}\times\left\{\frac{1}{2}\right\}\right) \subset R\,,$$where $\mathcal{G}_0$ is given by Proposition \ref{genorbdensas}, and $\mathbb{E}_{\infty}$ is given by Proposition \ref{evennumbers}.
\end{proof}

\section{Proof of Theorem \ref{ThmC}: admissible pairs for bi-critical circle maps}\label{S:real}

\subsection{Admissible pairs} We start Section \ref{S:real} with a definition. Remember that $R$ denotes the rectangle $[0,1]\times[-1,1]$ in $\R^2$, and $M=\big([0,1]\setminus\Q\big)\times[-1,1] \subset R$.

\begin{definition}\label{defadm} A pair $(\rho,\alpha)\in M$ is said to be \emph{admissible} if there exists a $C^{\infty}$ multicritical circle map $g$ with irrational rotation number $\rho$, a unique invariant measure $\mu$ and with exactly two critical points $c_1$ and $c_2$ such that the two connected components of $S^1\setminus\{c_1,c_2\}$ have $\mu$-measures equal to $\alpha$ and $1-\alpha$ respectively.
\end{definition}

The set of admissible pairs is denoted by $\mathbb{A}$.

\begin{lemma}\label{lemasquad} Any pair $(\rho,\alpha)\in(0,1)^2$ such that $\rho\notin\Q$ and $\rho-2\alpha=0$ belongs to $\mathbb{A}$.
\end{lemma}

\begin{proof}[Proof of Lemma \ref{lemasquad}] Let $f_0$ be a $C^{\infty}$ critical circle map with a single critical point $c(f_0)$ and such that $\rho(f_0)=\alpha$ (note that $f_0$ can be chosen to be real-analytic, say from the standard Arnold's family). Let us denote by $\mu$ the unique invariant Borel probability measure of $f_0$. Define $g=f_0^2=f_0 \circ f_0$\,, and note that $g$ is a bi-critical circle map, with irrational rotation number $\rho(g)=2\rho(f_0)=2\alpha=\rho$ and with two critical points $c_1(g)=c(f_0)$ and $c_2(g)=f_0^{-1}\big(c(f_0)\big)$. Moreover, the unique invariant Borel probability measure of $g$ is $\mu$, and the two connected components of $S^1\setminus\{c_1,c_2\}$
have $\mu$-measures equal to $\alpha$ and $1-\alpha$ respectively, since $c_1=f_0(c_2)$.
\end{proof}

The main result of this section is the following.

\begin{theorem}[The $C^{\infty}$ Realization Lemma]\label{lemaadm} Every pair in $M$ is admissible; in other words, $\mathbb{A}=M$.
\end{theorem}

The statement of Theorem \ref{lemaadm} is the same as the $C^{\infty}$ Realization Lemma given in the introduction. When combined with Theorem \ref{ThmB}, the $C^{\infty}$ Realization Lemma implies Theorem \ref{ThmC}. In order to prove Theorem \ref{lemaadm} we first remark the following consequence of Lemma \ref{leminhacomb}:

\begin{lemma}\label{lemaATinv} The set $\mathbb{A}$ of admissible pairs is forward invariant under $T$, where $T:M \to M$ is the skew product constructed in Section \ref{Secskew}. 
\end{lemma}

\begin{proof}[Proof of Lemma \ref{lemaATinv}] Let $(\rho,\alpha)\in\mathbb{A}$ and let $f$ be a $C^{\infty}$ bi-critical circle map, with critical points $c_1$ and $c_2$, such that $(\rho,\alpha)$ is the initial term of the renormalization trail of $c_2$ with respect to $c_1$ and $f$. For some fixed $n\in\nt$, we want to prove that $T^{n+1}(\rho,\alpha)\in\mathbb{A}$. By Lemma \ref{leminhacomb}, $T^{n+1}(\rho,\alpha)$ coincides with the $(n+1)$-th term $(\rho_{n+1},\alpha_{n+1})$ of the renormalization trail of $c_2$ (with respect to $c_1$ and $f$). Recall, from Section \ref{sectrails}, that $\rho_{n+1}=G^{n+1}(\rho)$ and that if $c_2$ belongs to the long interval $f^i\big(I_n(c_1)\big)$ for some $i\in\{0,1,...,q_{n+1}-1\}$, we have that$$\alpha_{n+1}=\frac{\mu\big((c_1,y_n)\big)}{\mu(I_n)}\,,$$where $y_n \in I_n(c_1)$ is given by $f^i(y_n)=c_2$. Otherwise, $c_2$ belongs to the short interval $f^j\big(I_{n+1}(c_1)\big)$ for some $j\in\{0,1,...,q_{n}-1\}$, and then$$\alpha_{n+1}=-\,\frac{\mu\big((y_n,c_1)\big)}{\mu(I_{n+1})}\,,$$where $y_n \in I_{n+1}(c_1)$ is given by $f^j(y_n)=c_2$. Let us assume that we are in the first case (the proof for the second one being the same), and note that the iterate $f^{q_n}$ restricts to a $C^{\infty}$ homeomorphism (with a critical point at $c_1$) between the intervals
$$
I_{n+1}(c_1) \cup f^{-q_{n+1}}\big(I_{n+1}(c_1)\big)=\big[f^{q_{n+1}}(c_1),f^{-q_{n+1}}(c_1)\big]\quad\mbox{and}
$$
$$\Delta_{0,n} \cup f^{-q_{n+1}}\big(\Delta_{0,n}\big)=\big[f^{q_{n+1}+q_n}(c_1),f^{-q_{n+1}+q_n}(c_1)\big],
$$
where $\Delta_{0,n}=f^{q_n}\big(I_{n+1}(c_1)\big)=\big(f^{q_{n+1}+q_n}(c_1),f^{q_n}(c_1)\big]$, as defined during the proof of Lemma \ref{leminhacomb}. Identifying points in this way we obtain from the interval$$I_{n+1}(c_1) \cup I_{n}(c_1) \cup f^{-q_{n+1}}\big(\Delta_{0,n}\big)=\big[f^{q_{n+1}}(c_1),f^{-q_{n+1}+q_n}(c_1)\big],$$a compact boundaryless one-dimensional topological manifold $N$. Denote by $\pi:I_{n+1}(c_1) \cup I_{n}(c_1) \cup f^{-q_{n+1}}\big(\Delta_{0,n}\big) \to N$ the quotient map, and let $\phi: N \to S^1$ be any homeomorphism which is a $C^{\infty}$ diffeomorphism between $N\setminus\{\pi(c_1)\}$ and $S^1\setminus\big\{\phi\big(\pi(c_1)\big)\big\}$. Note that $\phi\circ\pi$ maps the interior of $I_n(c_1)$ $C^{\infty}$-diffeomorphically onto $S^1\setminus\big\{\phi\big(\pi(c_1)\big)\big\}$. Let $g:S^1 \to S^1$ be given by the identity$$g \circ \phi \circ \pi=\phi \circ \pi \circ f^{q_{n+1}}\,\,\,\mbox{in $I_n(c_1)$,}$$and note that $g$ is a well-defined $C^{\infty}$ circle homeomorphism, with irrational rotation number equal to $\rho_{n+1}=G^{n+1}(\rho)$. Moreover, $g$ has exactly two critical points in $S^1$, given by $\hat{c}_1=\phi \circ \pi(c_1)$ and $\hat{c}_2=\phi \circ \pi(y_n)$. Finally, note that the unique invariant Borel probability measure $\mu_g$ of $g$ in $S^1$ is given by:$$\mu_g\big(\phi \circ \pi(A)\big)=\mu(A)/\mu\big(I_n(c_1)\big)=\mu(A)/\prod_{j=0}^{j=n}G^j(\rho)\,,$$for any Borel set $A \subset I_n(c_1)$. In particular, the two connected components of $S^1\setminus\{\hat{c}_1,\hat{c}_2\}$ have $\mu_g$\,-\,measures equal to $\alpha_{n+1}$ and $1-\alpha_{n+1}$ respectively. This finishes the proof of Lemma \ref{lemaATinv}.
\end{proof}

We remark that the \emph{glueing procedure} described in the proof of Lemma \ref{lemaATinv} was introduced by Lanford in the eighties, see \cite{lanford1,lanford2} for much more.

\medskip

Since the set $\mathbb{A}$ of admissible pairs is obviously non-empty (see for instance Lemma \ref{lemasquad} above), Theorem \ref{lemaadm} follows by combining Lemma \ref{Tcov} and
Lemma \ref{lemaATinv} with the following result.

\begin{prop}\label{lemaAab} The set $\mathbb{A}$ of admissible pairs has non-empty interior in $M$.
\end{prop}

In order to prove Proposition \ref{lemaAab}, we need some preliminary constructions. Let $f$ be a smooth multicritical
circle map with irrational rotation number $\rho_f$, a unique invariant
Borel probability measure $\mu_f$ and with exactly two critical points $c_1$ and $c_2$ such that the two connected
components of $S^1\setminus\{c_1,c_2\}$ have $\mu_f$-measures equal to $\alpha_f$ and $1-\alpha_f$ respectively.
Denote by $\Delta_f$ the one whose measure equals $\alpha_f$, that is:$$\alpha_f=\int_{\Delta_f}\!d\mu_f\,.$$

By Birkhoff's Ergodic Theorem (combined with the unique ergodicity of $f$), we can write:
$$\alpha_f=\lim_{n\to+\infty}\left\{\frac{1}{n}\sum_{j=0}^{n-1}\chi_{\Delta_f}\big(f^{j}(x)\big)\right\}\quad\mbox{for any $x \in S^1$}\,,$$
where $\chi_{\Delta_f}$ is the \emph{characteristic} function of the open interval $\Delta_f$. By the well-known \emph{Denjoy-Koksma inequality} (see \cite[p.~73]{hermanIHES}), we have for any $x \in S^1$ and any $n\in\nt$ that:
\begin{equation}\label{DK}
\left|\alpha_f-\frac{1}{q_{n}}\sum_{j=0}^{q_{n}-1}\chi_{\Delta_f}\big(f^{j}(x)\big)\right|\leq\frac{\varia(\chi_{\Delta_f})}{q_n}=\frac{2}{q_n}\,,
\end{equation}
where, as before, $\{q_n\}_{n\in\nt}$ is the sequence of return times given by $\rho_f$, the rotation number of $f$.
Both parameters $\rho_f$ and $\alpha_f$ are continuous under $C^1$ perturbations. More precisely, we have the following lemma.

\begin{lemma}\label{lemacont} Given $\varepsilon>0$ there exists $\delta=\delta(\varepsilon,f)>0$ such that
if $g$ is a smooth multicritical circle map with irrational rotation number $\rho_g$, with exactly two critical points and
satisfying $d_{C^1}(f,g)<\delta$, then $|\rho_f-\rho_g|<\varepsilon$ and $|\alpha_f-\alpha_g|<\varepsilon$.
\end{lemma}

It is well-known that the rotation number is continuous under $C^0$ perturbations, so the main point in the proof of Lemma \ref{lemacont} is to establish the continuity of $\alpha_f$.

\begin{proof}[Proof of Lemma \ref{lemacont}] Let $n_0\in\nt$ be large enough that $q_{n_0}>4/\varepsilon$, and
let $x \in S^1$ be such that $f^{j}(x)$ is a regular point of $f$ for all $j\in\{0,1,...,q_{n_0}-1\}$. Choose $\delta>0$ small enough in order to have the following property: if $g$ is a smooth multicritical circle map with irrational rotation number $\rho_g$, with exactly two critical points and satisfying $d_{C^1}(f,g)<\delta$, then:
\begin{itemize}
\item If $\rho_f=[a_0,a_1,...,a_{n_0},a_{n_0+1},...]$, then $\rho_g=[a_0,a_1,...,a_{n_0},b_{n_0+1},...]$; in particular, we have  $q_{n_0}(\rho_g)=q_{n_0}(\rho_f)=q_{n_0}$.
\item For all $j\in\{0,1,...,q_{n_0}-1\}$ we have: $\displaystyle f^{j}(x)\in\Delta_f \Leftrightarrow g^{j}(x)\in\Delta_g$\,.
\end{itemize}

Applying estimate \eqref{DK} above we obtain:
\begin{align*}
\big|\alpha_{f}-\alpha_{g}\big|&\leq\left|\alpha_{f}-\frac{1}{q_{n_0}}\sum_{j=0}^{q_{n_0}-1}\chi_{\Delta_f}\big(f^{j}(x)\big)\right|+\left|\alpha_{g}-\frac{1}{q_{n_0}}\sum_{j=0}^{q_{n_0}-1}\chi_{\Delta_g}\big(g^{j}(x)\big)\right|\\
&+\frac{1}{q_{n_0}}\sum_{j=0}^{q_{n_0}-1}\big|\chi_{\Delta_f}\big(f^{j}(x)\big)-\chi_{\Delta_g}\big(g^{j}(x)\big)\big|\\
&=\left|\alpha_{f}-\frac{1}{q_{n_0}}\sum_{j=0}^{q_{n_0}-1}\chi_{\Delta_f}\big(f^{j}(x)\big)\right|+\left|\alpha_{g}-\frac{1}{q_{n_0}}\sum_{j=0}^{q_{n_0}-1}\chi_{\Delta_g}\big(g^{j}(x)\big)\right|\\
&\leq\frac{4}{q_{n_0}}<\varepsilon\,.
\end{align*}
\end{proof}

\subsection{Two-parameter families}
Roughly speaking, the key to proving Proposition \ref{lemaAab} is to show that the set $\mathbb{A}$ of admissible pairs $(\rho,\alpha)$
intersects the fiber above each irrational number $\rho\in (0,1)$ in a ``long'' interval $J_\rho$.
Thus, we need for each such $\rho$ a (continuous) one-parameter family $\mathcal{G}_{\rho}$ of bi-critical circle maps
such that, for each $f\in \mathcal{G}_{\rho}$, we have $\rho(f)=\rho$ and $\{(\rho,\alpha_f):\, f\in \mathcal{G}_{\rho}\,\}= J_\rho$.
In order to accomplish this goal, we first build for each $\rho$ a special two-parameter family of bi-critical homeomorphisms
of the circle from which $\mathcal{G}_{\rho}$ will be extracted.

Let us start by fixing $\rho_0\in (0,1)\setminus \mathbb{Q}$. Let $a>0$ and $\delta>0$ be both much smaller than $\rho_0$ (how small they have to be will be determined in the course of the arguments). Let $\varphi_0: [-1,1]\to \mathbb{R}$ be a smooth function having the following properties:

\begin{itemize}
 \item $\mathrm{supp}(\varphi_0)\subset [-a/2,a/2]$;
 \item $\|\varphi_0\|_{C^0} = \delta$;
 \item $\varphi_0'(0)=-1$ and $|\varphi_0'(x)|<1$ for all $x\in [-1,1]\setminus \{0\}$;
 \item $\varphi_0''(0)\neq 0$.
\end{itemize}

The construction of a function $\varphi_0$ with these properties is an exercise using standard bump functions.
Now extend $\varphi_0$ so as to make it into a $\mathbb{Z}$-periodic function $\varphi:\mathbb{R}\to \mathbb{R}$,
{\it i.e.,} set $\varphi(x+n)=\varphi_0(x)$ for all $x\in [-1,1]$ and all $n\in \mathbb{Z}$.

Next, for $a\leq t\leq 1-a$ and $-\epsilon\leq s\leq \epsilon$, where $\epsilon>2\delta$ is still much smaller than $\rho_0$, define
$\widetilde{f}_{t,s}: \mathbb{R}\to \mathbb{R}$ by
\[
 \widetilde{f}_{t,s}(x)\;=\; x+ \rho_0+s+\varphi(x)+\varphi(x-t)\ .
\]
Then $\widetilde{f}_{t,s}$ is a smooth, orientation-preserving homeomorphism whose set of critical
points is equal to $\mathbb{Z}\cup (t+\mathbb{Z})$. By the last of the conditions on $\varphi_0$ above,
each critical point is non-flat. The quotient map $f_{t,s}$ on the circle (via the exponential map
$x\mapsto \exp(2\pi i x)$) is a bi-critical circle map (its critical points being $1=\exp(0)$ and $\exp(2\pi it)$).

Let us define $\mathcal{G}_{\rho_0}=\{(t,s)\in [a,1-a]\times[-\epsilon,\epsilon]\,:\; \rho(f_{t,s})=\rho_0\}$.

\begin{lemma}\label{lemGisagraph}
 The set $\mathcal{G}_{\rho_0}$ is the graph of a continuous function $t\mapsto \psi(t)$ defined
 on the interval $[a,1-a]$.
\end{lemma}

\begin{proof}[Proof of Lemma \ref{lemGisagraph}]
 We divide the proof into two steps.
 \begin{enumerate}
  \item[(i)] {\emph{$\mathcal{G}_{\rho_0}$ is a graph.\/}}
  Note that for each $t\in [a,1-a]$ we have
  \[
  \widetilde{f}_{t,-2\delta}(x) < x+\rho_0-\delta\ , \ \ \textrm{for all}\ x\in \mathbb{R}\ ,
  \]
  whereas
  \[
  \widetilde{f}_{t,+2\delta}(x) > x+\rho_0+\delta\ , \ \ \textrm{for all}\ x\in \mathbb{R}\ .
  \]
  This shows that $\rho(f_{t,-2\delta})\leq \rho_0 -\delta$, whereas $\rho(f_{t,+2\delta})\geq \rho_0 +\delta$.
  Hence, by {\emph{continuity and monotonicity}} of $\rho(f_{t,s})$ as a function of $s$ (for each fixed $t$), it follows that
  there exists a {\emph{unique}} $s_t\in (-2\delta,+2\delta)$ such that $\rho(f_{t,s_t})=\rho_0$. We define
  $\psi(t)=s_t$. Thus we have proved that $\mathcal{G}_{\rho_0}=\{(t,\psi(t)): \, a\leq t\leq 1-a\}= \mathrm{Gr}(\psi)$.

  \item[(ii)] {\emph{The function $t\mapsto \psi(t)$ is continuous.\/}} As is well known, the rotation
  number $\rho(f_{t,s})$ is continuous as a function from $[a,1-a]\times[-\epsilon,\epsilon]$ to $\mathbb{R}/\mathbb{Z}$.
  In addition, we obviously have $\mathcal{G}_{\rho_0}=\rho^{-1}(\rho_0)$. Since $\{\rho_0\}\subset \mathbb{R}/\mathbb{Z}$
  is closed, so is $\mathcal{G}_{\rho_0}$. Hence the graph of $t\mapsto \psi(t)$ is closed, and this means
  precisely that $\psi(t)$ is continuous.
 \end{enumerate}

\end{proof}

\begin{lemma}\label{lemverticals}
 If $0<\rho_0<\frac{1}{8}$ is irrational, then $J_{\rho_0} \supset [2\rho_0,3\rho_0]$.
\end{lemma}

\begin{proof}[Proof of Lemma \ref{lemverticals}]
 We use the family $\{\widetilde{f}_{t,s}\}$ introduced above. This family also depends on the choice
 of the positive numbers $a$ and $\delta$; we shall
 make $a$ and $\delta$ as small as needed for the argument that follows to work.

 For each $a\leq t\leq 1-a$, let $s=s_t$, where $s_t$ is as in the proof of Lemma \ref{lemGisagraph}.
 Let us write $\phi_t=\widetilde{f}_{t,s_t}$, so that
 \[
  \phi_t(x)\;=\; x+\rho_0 + s_t + \varphi(x) + \varphi(x-t)\ .
 \]
 We henceforth identify, by an abuse of notation, $\phi_t:\mathbb{R}\to \mathbb{R}$
 with its quotient map on the circle $\mathbb{R}/\mathbb{Z}$. Note that, by construction, each $\phi_t$
 has rotation number equal to $\rho_0$.
 Since the functions $\varphi(x)$ and $\varphi(x-t)$ have disjoint supports and $C^0$ norms bounded by $\delta$,
 we have $|\varphi(x)+\varphi(x-t)|\leq \delta$. Also, by construction we have $s_t\in (-2\delta, 2\delta)$.
 Hence we see that
 \begin{equation}\label{onephit}
  x+\rho_0 -3\delta \;<\; \phi_t(x) \;<\; x+ \rho_0 + 3\delta\ .
 \end{equation}
From \eqref{onephit} it follows by induction that, for all $k \geq 0$,
\begin{equation}\label{twophit}
 x+k\rho_0 -3\delta k \;<\; \phi_t^k(x) \;<\; x+ k\rho_0 + 3\delta k\ .
\end{equation}
Next, let $\mu_t$ denote the unique Borel probability measure on $\mathbb{R}/\mathbb{Z}$ which is invariant under
$\phi_t$. Recall that, on the circle, the points $c_0=0$ and $c_t=t$ are the two critical points of $\phi_t$.
We will use \eqref{twophit} to estimate the measure of the segment $[c_0,c_t]$, {\it i.e.,\/}
$\alpha_t = \mu_t[0,t]$. The basic observation is that for each $x$ and each $k$, the fundamental domain
$[\phi_t^{k-1}(x)\,,\,\phi_t^k(x)]$ has $\mu_t$-measure equal to $\rho_0$. Now, there exists a unique $m\geq 1$
such that $[0,\phi_t^{m-1}(0)] \subseteq [0,t] \subset [0,\phi_t^m(0)]$. From these facts it follows that
\begin{equation}\label{threephit}
 (m-1)\rho_0 \;<\; \alpha_t \;<\; m\rho_0 \ .
\end{equation}
Also, we obviously have $\phi_t^{m-1}(0)\leq t<\phi_t^m(0)$. Using \eqref{twophit} with $k=m-1$ and with $k=m$,
we get
\[
 (m-1)(\rho_0-3\delta) \;\leq\; t \;<\; m(\rho_0+3\delta)\ .
\]
We rewrite this as
\begin{equation}\label{fourphit}
 \frac{t}{\rho_0+3\delta} \;<\; m \;\leq\; 1 + \frac{t}{\rho_0-3\delta} \ .
\end{equation}
Putting \eqref{fourphit} back into \eqref{threephit} , we get
\begin{equation}\label{fivephit}
 \rho_0 \left(\frac{t}{\rho_0+3\delta} - 1\right) \;<\; \alpha_t \;<\; \rho_0 \left( 1 + \frac{t}{\rho_0-3\delta}\right)\ .
\end{equation}
Now we have two special cases to consider, namely $t=a$ and $t=1-a$. In the first case, using the second inequality in
\eqref{fivephit} we have
$\alpha_a<2\rho_0$, provided $a$ and $\delta$ are so small that $a/(\rho_0-3\delta) < 1$.
In the second case, the first inequality in \eqref{fivephit} tells us that
\begin{equation}\label{sixphit}
 \alpha_{1-a} \;>\; \left(\frac{1-a}{\rho_0+3\delta} - 1\right) \rho_0 \ .
\end{equation}
It is straightforward to see that the right-hand side of \eqref{sixphit} will be $> 3\rho_0$ provided
$0<\rho_0<\frac{1}{8}$ and we take $a<\frac{1}{4}$ and $\delta<\frac{1}{48}$.
Summarizing, we have proved that $\alpha_a<2\rho_0 < 3\rho_0 < \alpha_{1-a}$
(provided $a$ and $\delta$ are sufficiently small). But by Lemma \ref{lemacont}, the function $t\mapsto \alpha_t$ is continuous.
Hence its image certainly contains the interval $[2\rho_0,3\rho_0]$. This proves that $J_{\rho_0}\supset [2\rho_0,3\rho_0]$, and
we are done.
\end{proof}

Proposition \ref{lemaAab} is an immediate consequence of this last lemma.

\begin{proof}[Proof of Proposition \ref{lemaAab}]
 By Lemma \ref{lemverticals}, we have
 \[
  \mathbb{A} \;\supset\; \bigcup_{\rho_0\in \left[\frac{1}{9},\frac{1}{8}\right]\setminus \mathbb{Q}} \{\rho_0\}\times [2\rho_0,3\rho_0]
  \;\supset\; \left( \left(\frac{1}{9},\frac{1}{8}\right)\setminus \mathbb{Q}\right) \times \left(\frac{1}{4}, \frac{1}{3}\right)\ .
 \]
Since this last rectangle is open in $M$, it follows that, indeed, $\mathbb{A}$ has non-empty interior in $M$.
\end{proof}

\appendix

\section{The skew-product $T$ is ergodic: proof of Proposition \ref{genorbdensas}}\label{apperg}

In  Section \ref{Secskew}, we considered the skew-product $T:M\to M$. 
Here, we enlarge it to get a self-map of the rectangle  $R=[0,1]\times [-1,1]$. It suffices to define the fiber maps $T_\rho:[-1,1]\to [-1,1]$ also for rational values of $\rho$. When $\rho\in [0,1]\cap \mathbb{Q}$ is not of the form $\rho=\frac{1}{n}$, we define $T_\rho$ using the same formulas given in \ref{SecFibermaps}. We also 
define $T_0\equiv 0$, and for each $n\in \mathbb{N}$, $T_{1/n}: [-1,1]\to [-1,1]$ by $T_{1/n}(\alpha)=-\alpha$ if $\alpha\in [-1,0]$ and 
$T_{1/n}(\alpha)=\{n(1-\alpha)\}$ if $\alpha\in (0,1]$. 
Hence we can define the extended skew-product $T:R\to R$  
by $T(\rho,\alpha)=(G(\rho), T_\rho(\alpha))$, where as before $G:[0,1]\to [0,1]$ is the Gauss map, and for each $\rho\in [0,1]$, $T_\rho:[-1,1]\to [-1,1]$. We note {\it en passant\/} that the composition of any two of these fiber maps (with $\rho\neq 0$) is expanding.

\medskip

Our main purpose in this appendix is to prove the following result.

\begin{theorem}\label{skewprodthm} The skew-product $T:R\to R$ admits a unique invariant Borel probability measure which is absolutely continuous with respect to the Lebesgue measure. This invariant measure is ergodic under $T$, and its support coincides with $R$.
\end{theorem}

Contrary to what happens for one-dimensional maps, a piecewise smooth two-di\-men\-sional expanding map may not admit an absolutely continuous invariant measure; additional hypotheses are necessary (see for instance \cite{Buzzi, Ts2} and references therein). However, in our case the map $T$ is rather special. The fact that $T$ is a skew-product, combined with the fact that it is a Markov map (see below) which is affine on the fibers, allows us to reduce the problem to an essentially one-dimensional situation. Indeed, we start this appendix with the following useful property of the family of fiber maps defined in Section \ref{SecFibermaps}.

\begin{lemma}\label{lemadistTrho} Given any sequence $\{\theta_n\}_{n\in\nt}\subset[0,1]\setminus\Q$\, consider the sequence of compositions $\big\{\Psi_{\theta_{0}\cdots\,\theta_{n-1}}\big\}_{n \geq 1}$ in $[-1,1]$ given by:$$\Psi_{\theta_{0}\cdots\,\theta_{n-1}}=T_{\theta_0} \circ T_{\theta_1} \circ ... \circ T_{\theta_{n-1}}\quad\mbox{for all $n \geq 1$.}$$

Then for any given Borel set $B\subset[-1,1]$, the sequence $\big\{\lambda\big(\Psi_{\theta_{0}\cdots\,\theta_{n-1}}^{-1}(B)\big)\big\}_{n\in\nt}$ is convergent, where $\lambda$ denotes the Lebesgue measure on $[-1,1]$. Moreover:$$\theta_0\,G(\theta_0)\,\lambda(B)\leq\lim_{n\to+\infty}\big\{\lambda\big(\Psi_{\theta_{0}\cdots\,\theta_{n-1}}^{-1}(B)\big)\big\}\leq\big(2-\theta_0\,G(\theta_0)\big)\,\lambda(B)\,.$$
\end{lemma}

\begin{proof}[Proof of Lemma \ref{lemadistTrho}] From the given sequence $\{\theta_n\}$ consider the sequence $\{\tau_n\}_{n\in\nt}\subset[0,1]$ given by:$$\tau_0=1,\,\,\tau_1=0\quad\mbox{and}\quad\tau_{n+2}=\theta_n\,G(\theta_n)\,\tau_n+\big(1-\theta_n\,G(\theta_n)\big)\,\tau_{n+1}\quad\mbox{for all $n\in\nt$.}$$
In other words, $\tau_2=\theta_0\,G(\theta_0)$ and:$$\tau_n=\theta_0\,G(\theta_0)+\sum_{j=1}^{n-2}(-1)^{j}\,\prod_{i=0}^{i=j}\theta_i\,G(\theta_i)\quad\mbox{for all $n \geq 3$.}$$
The sequence $\{\tau_n\}$ clearly converges to some number $\tau_{\infty}$, which satisfies\footnote{Remember here that $\theta\,G(\theta)\in(0,1/2)$ for any $\theta\in[0,1]\setminus\Q$ (if $\theta<1/2$ this is obvious since $0<G(\theta)<1$; if $\theta>1/2$, then $\theta\,G(\theta)=1-\theta$).}:$$0<\frac{\theta_0\,G(\theta_0)}{2}<\tau_{\infty}<\theta_0\,G(\theta_0)<\frac{1}{2}\,.$$
Given a Borel set $B\subset[-1,1]$ and $n\in\nt$ let $\ell_n$ and $r_n$ in $[0,1]$ be given by:$$\ell_n=\lambda\big(\Psi_{\theta_{0}\cdots\,\theta_{n-1}}^{-1}(B)\cap[-1,0]\big)\quad\mbox{and}\quad r_n=\lambda\big(\Psi_{\theta_{0}\cdots\,\theta_{n-1}}^{-1}(B)\cap[0,1]\big).$$
By definition of each $T_{\theta}$, the following relations hold for all $n\in\nt$:
\[
\begin{dcases}
\ell_{n+1}=r_n\\[0.4ex]
r_{n+1}=\theta_n\,G(\theta_n)\,\ell_n+\left\lfloor\frac{1}{\theta_n}\right\rfloor\theta_n\,r_n=\theta_n\,G(\theta_n)\,\ell_n+\big(1-\theta_n\,G(\theta_n)\big)\,r_n\\
\end{dcases}
\]
With this at hand, we easily obtain by induction that for all $n\in\nt$:
\[
\begin{dcases}
\ell_n=\tau_n\,\ell_0+(1-\tau_n)\,r_0\\[0.4ex]
r_n=\tau_{n+1}\,\ell_0+(1-\tau_{n+1})\,r_0\\
\end{dcases}
\]
In particular, the Lebesgue measure of $\Psi_{\theta_{0}\cdots\,\theta_{n-1}}^{-1}(B)$ in $[-1,1]$ is given by:$$\lambda\big(\Psi_{\theta_{0}\cdots\,\theta_{n-1}}^{-1}(B)\big)=(\tau_n+\tau_{n+1})\,\ell_0+\big(2-(\tau_n+\tau_{n+1})\big)\,r_0,$$which converges to $2\,\big(\tau_{\infty}\,\ell_0+(1-\tau_{\infty})\,r_0\big)$ as $n$ goes to infinity. This proves Lemma \ref{lemadistTrho}.
\end{proof}

With Lemma \ref{lemadistTrho} at hand we have the following result. 

\begin{lemma}\label{lemaacip} The skew product $T$ preserves a probability measure $\mu_T$ on the rectangle $R$ which is absolutely continuous (with respect to Lebesgue).
\end{lemma}

\begin{proof}[Proof of Lemma \ref{lemaacip}] We only sketch the arguments, as they are quite standard. As before, denote by $\nu$ and $\lambda$ the Gauss measure on $[0,1]$ and the Lebesgue measure on $[-1,1]$ respectively. Denote by $\mu$ the absolutely continuous (with respect to Lebesgue) Borel measure on the rectangle $R$ given by\, $\mu=\nu\times\lambda$\,. In other words, given a Borel set $A \subset R$ we have:$$\mu(A)=\int_{\pi_1(A)}\!\lambda_{\rho}(A)\,\,d\nu(\rho)\,,$$where $\pi_1:R\to[0,1]$ is the projection on the first coordinate given by $\pi_1(\rho,\alpha)=\rho$, and where $\lambda_{\rho}$ is the Lebesgue measure on the vertical fiber given by $\rho$, that is: $\lambda_{\rho}(A)=\lambda\big(A\cap(\{\rho\}\times[-1,1])\big)$ for any $\rho\in[0,1]$.

Given $n\in\nt$ and open intervals $I\subset[0,1]$ and $J\subset[-1,1]$ we label each point $\theta_{n-1}$ of $G^{-n}(I)$ with the $n$-tuple $\{\theta_0,...,\theta_{n-1}\}$ given by $G(\theta_0) \in I$ and $G(\theta_{i})=\theta_{i-1}$ for all $i\in\{1,...,n-1\}$. With this notation we can write:$$T^{-n}(I \times J)=\bigcup_{\substack{\{\theta_0,...,\theta_{n-1}\}\\G(\theta_0) \in I\,,\,G(\theta_{i})=\theta_{i-1}}}\big\{\theta_{n-1}\big\}\times\Psi_{\theta_{0}\cdots\,\theta_{n-1}}^{-1}(J)\,.$$
From Lemma \ref{lemadistTrho} we know that$$\lambda\big(\Psi_{\theta_{0}\cdots\,\theta_{n-1}}^{-1}(J)\big) \leq 2\,\lambda(J)$$holds for any $n$-tuple, and then:
\begin{align*}
\mu\big(T^{-n}(I \times J)\big)&=\int_{G^{-n}(I)}\!\lambda_{\rho}\big(T^{-n}(I \times J)\big)\,d\nu(\rho)\leq 2\,\lambda(J)\int_{G^{-n}(I)}\!d\nu(\rho)\\
&=2\,\lambda(J)\,\nu\big(G^{-n}(I)\big)=2\,\lambda(J)\,\nu(I)=2\,\mu(I \times J).
\end{align*}
With this at hand we deduce that:
\begin{equation}\label{eqestpullback}
\big(T^{n}_{*}\mu\big)(A) \leq 2\,\mu(A)\quad\mbox{for any Borel set $A \subset R$ and any $n\in\nt$.}
\end{equation}

Finally, consider the sequence of Borel measures on the rectangle $R$ given by$$\mu_n=\frac{1}{n}\sum_{j=0}^{n-1}T^{j}_{*}\mu\,.$$Since $T$ is a local diffeomorphism around Lebesgue almost every point in $R$, we deduce that the push-forward under $T$ of any absolutely continuous measure is also absolutely continuous and that, when restricted to absolutely continuous measures, the operator $T_{*}$ acts continuously in the weak* topology. Let $\omega$ be any weak* accumulation point of $\{\mu_n\}$ (recall that $\mu_n(R)=2$ for all $n$). By \eqref{eqestpullback}, $\omega(A) \leq 2\,\mu(A)$ for any Borel set $A \subset R$. Therefore, $\omega$ is absolutely continuous with respect to $\mu$, and then it is also absolutely continuous with respect to Lebesgue. In particular, the measure $\omega$ is a continuity point of $T_{*}$, which implies that it is $T$-invariant in the usual way. We conclude the proof of Lemma \ref{lemaacip} by taking the \emph{probability} measure $\mu_T=\frac{1}{2}\,\omega$.
\end{proof}

In order to prove Theorem \ref{skewprodthm}, it remains to prove that the absolutely continuous invariant probability measure $\mu_T$ given by Lemma \ref{lemaacip} is unique, supported on the whole rectangle $R$ and ergodic under $T$ (see Corollary \ref{coroergodicity} below).

\subsubsection*{A countable Markov partition} The skew-product $T$ admits a countable Markov partition that we presently describe. The basic (open) Markov atoms of the partition are of three different types (see figure \ref{tiles}):

\begin{enumerate}
\item The trapezoids $V_{k,\ell}$, with $k\in \mathbb{N}$ and $0\leq \ell \leq k-1$, given by 
\[
 V_{k,\ell} = \left\{(\rho,\alpha)\in R:\; \frac{1}{k+1}< \rho< \frac{1}{k}\;,\;1-(\ell+1)\rho<\alpha<1-\ell\rho\right\} \ ;
\]
 \item The triangles 
 \[
  U_k = \left\{(\rho,\alpha)\in R:\; \frac{1}{k+1}< \rho< \frac{1}{k}\;,\;0<\alpha<1-k\rho\right\}\ \ \ (k\in \mathbb{N})\ ;
 \]
 \item The rectangles 
 \[
  R_k=\left\{(\rho,\alpha)\in R:\; \frac{1}{k+1}< \rho< \frac{1}{k}\;,\;-1<\alpha<0\right\}\ \ \ (k\in \mathbb{N})\ .
 \]
\end{enumerate}

\begin{figure}[t]
\begin{center}~
\hbox to \hsize{\psfrag{0}[][][1]{$0$} \psfrag{x}[][][1]{$x$}
\psfrag{a}[][][1]{$\alpha$}
\psfrag{p}[][][1]{$\!\!\!\rho$}
\psfrag{1}[][][1]{$1$}
\psfrag{-1}[][][1]{$\!\!\!\!\!\!-1$}
\psfrag{k}[][][1]{$\frac{1}{k}$}
\psfrag{k1}[][][1]{$\!\!\!\!\frac{1}{k+1}$}
\psfrag{U}[][][1]{$U_k$}
\psfrag{V}[][][1]{$\;V_{k,\ell}$}
\psfrag{R}[][][1]{$R_k$}
\psfrag{0}[][][1]{$\!\!0$}
\psfrag{c}[][][1]{$\mathbf{\cdots}$}
\hspace{1.0em} \includegraphics[width=3.5in]{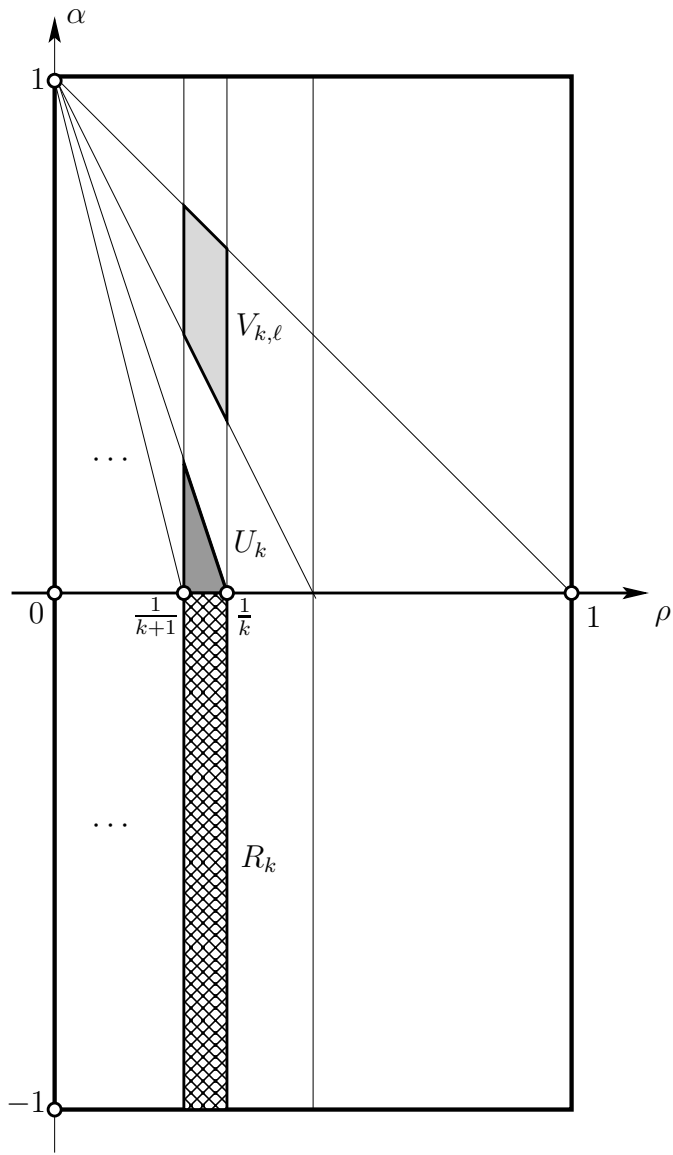}
   }
\end{center}
  \caption[markovtiles]{\label{tiles} The Markov partition for $T$ has three different types of atoms.}
\end{figure}

The map $T$ is one-to-one in each of these Markov atoms, mapping them diffeomorphically 
onto either $R^+=(0,1)\times (0,1)$ or 
$R^-=(0,1)\times (-1,0)$. More precisely, we have $T(U_k)=R^-$, $T(R_k)=R^+$ 
and $T(V_{k,\ell})=R^+$, for all $k$ and all $\ell$. 
The collection $\mathcal{P}$ of all such atoms is our {\it Markov partition\/} for $T$.

\subsubsection*{Markov tiles} Let us write 
\[
 \mathcal{P} = \left\{ W_1,W_2,\ldots,W_m,\ldots\right\}
\]
for an enumeration of the elements of the Markov partition $\mathcal{P}$. For each $m$, let 
$\tau_m:R^{\pm}\to W_m$ be the inverse branch of $T$  that takes $T(W_m)=R^{\pm}$ back onto 
$W_m$. Then $\tau_m$ is a smooth diffeomorphism and we have 
$\tau_m\circ T=\mathrm{id}_{W_m}$ and $T\circ \tau_m=\mathrm{id}_{R^{\pm}}$. 
An $n$-tuple $(m_1,m_2,\ldots,m_n)\in \mathbb{N}^n$ is said to be {\it admissible\/} if the composition $\tau_{m_1}\circ \tau_{m_2}\circ \cdots \tau_{m_n}$ is well-defined (as a map of $R^{\pm}$ into $R$). For each admissible $n$-tuple $(m_1,m_2,\ldots,m_n)\in \mathbb{N}^n$, we consider the region (polygon)
\[
 W_{m_1,m_2,\ldots,m_n}=\tau_{m_1}\circ \tau_{m_2}\circ \cdots \circ\tau_{m_n}(R^{\pm})\ .
\]
Such region is called a {\it Markov $n$-tile\/}. Note that 
$T(W_{m_1,m_2,\ldots,m_n})=W_{m_2,\ldots,m_n}$, so each Markov $n$-tile is mapped onto a Markov $(n-1)$-tile if $n\geq 2$, or onto $R^{\pm}$ if $n=1$. 

\begin{lemma}\label{markovtiles1}
 There exist constants $C>0$ and $0<\lambda <1$ such that, for every 
 Markov $n$-tile $W_{m_1,m_2,\ldots,m_n}$, we have 
 \[\mathrm{diam}(W_{m_1,m_2,\ldots,m_n})<C\lambda^n \ .\]
\end{lemma}

\begin{proof}[Proof of Lemma \ref{markovtiles1}] This follows at once from the easily verifiable fact that the map $T^2=T\circ T$ is 
 expanding. 
\end{proof}

We denote by $\mathcal{W}$ the collection of all Markov tiles, and for each $n$ we denote by $\mathcal{W}^{(n)}$ the collection of all Markov $n$-tiles, so that $\mathcal{W}=\bigcup_{n\in\mathbb{N}}\mathcal{W}^{(n)}$. The following easily proven facts are worth keeping in mind here:
\begin{enumerate}
\item[(MT1)] For each $n$ the elements of $\mathcal{W}^{(n)}$ are pairwise disjoint open subsets of $R$;
\item[(MT2)] For each $n$ the complement of $\bigcup_{W\in \mathcal{W}^{(n)}}W$ 
in $R$ is a Lebesgue null-set;
\item[(MT3)] The union $\bigcup_{W\in \mathcal{W}} \partial W$ is a Lebesgue null-set;
\item[(MT4)] For each open subset $A\subseteq R$, there exists a collection 
$\mathcal{C}_A\subseteq \mathcal{W}$ of pairwise disjoint Markov tiles such 
that $A\setminus \bigcup_{W\in \mathcal{C}_A} W$ has zero Lebesgue measure. 
\end{enumerate}

Note that Lemma \ref{Tcov} follows at once from the fact that any given open set in $R$ contains the closure of an $n$-tile (and then it eventually covers the whole rectangle under iteration of $T$).

\subsubsection*{Bounding Jacobian distortion} One path towards proving ergodicity of $T$ is to show that the Jacobians of all inverse branches of iterates of $T$ have uniformly bounded distortion. This 
follows from Proposition \ref{boundjacobdist} below. In the proof, we will need the following  simple lemma. 

\begin{lemma}\label{simplelemma} Let $k_j>0$, $b_j\geq 0$ ($j\geq 0$) be two sequences of real numbers, and assume that $B=\sum_{j=0}^{\infty} \sqrt{b_j}<\infty$. Then for each $n\in \mathbb{N}$ we have
 \begin{equation}\label{lamelem}
  \sum_{j=0}^{n} k_j\min\left\{b_{n-j}\,,\,{k_j}^{-2}\right\} \leq B\ .
 \end{equation}
\end{lemma}

\begin{proof}[Proof of Lemma \ref{simplelemma}] For each $1\leq j\leq n$, there are only two possibilities:
 \begin{enumerate}
  \item[(i)] $k_j^{-2}< b_{n-j}$: In this case we have
  \[
  k_j\min\{b_{n-j},k_j^{-2}\} = k_j^{-1} < \sqrt{b_{n-j}}\ .
  \]
  
  \item[(ii)] $k_j^{-2}\geq  b_{n-j}$: In this case we have
  \[
  k_j\min\{b_{n-j},k_j^{-2}\} = k_j b_{n-j} \leq (b_{n-j}^{-1})^{\frac{1}{2}}b_{n-j} = \sqrt{b_{n-j}}\ .
  \]
  
 \end{enumerate}
 From (i) and (ii) it follows that the sum in the left-hand side of \eqref{lamelem} 
 is bounded by $\sum_{j=0}^{n} \sqrt{b_{n-j}} \leq B$. 

\end{proof}

\begin{prop}\label{boundjacobdist}
 There exists a constant $K>1$ for which the following holds for all 
 $n\in \mathbb{N}$. If $(\rho_0,\alpha_0)$ and $(\rho_0^*,\alpha_0^*)$ are any two points in the same 
 Markov $n$-tile, then 
 \begin{equation}\label{jacobratio}
  \frac{1}{K}\leq \left|\frac{\det DT^n(\rho_0,\alpha_0)}{\det DT^n(\rho_0^*,\alpha_0^*)}\right| \leq K\ .
 \end{equation}
\end{prop}

\begin{proof}[Proof of Proposition \ref{boundjacobdist}] First, some preliminary considerations. 
 For definiteness, let $W_{m_1,m_2,\ldots,m_n}$ be the Markov $n$-tile containing the two points $(\rho_0,\alpha_0)$ and $(\rho_0^*,\alpha_0^*)$. Let us write, for 
 $j=1,2,\ldots$, $(\rho_j,\alpha_j)=T^j(\rho_0,\alpha_0)$ and $(\rho_j^*,\alpha_j^*)=T^j(\rho_0^*,\alpha_0^*)$. From the definition of our skew-product, we see that
 \begin{equation}\label{xyskew}
  \left\{\begin{array}{lll}
          \rho_j &= & G^j(\rho_0) \\
          \alpha_j &= & T_{\rho_{j-1}}\circ T_{\rho_{j-2}}\circ \cdots\circ T_{\rho_0}(\alpha_0)
         \end{array}
 \right. \,
 \end{equation}
 and similar formulas hold for $\rho_j^*,\alpha_j^*$. 
Note also that  
 $(\rho_j,\alpha_j),(\rho_j^*,\alpha_j^*) \in W_{m_{j+1},\ldots ,m_n}$ for each $0\leq j\leq n$. Hence, by Lemma \ref{markovtiles1}, for each such $j$ we have
 \[
  |\rho_j-\rho_j^*| \leq \mathrm{diam}(W_{m_{j+1},\ldots,m_n}) < C\lambda^{n-j}
 \]
 Next, for each $0\leq j\leq n$, let $k_j$ be the unique natural number such that $\frac{1}{k_j+1}<\rho_j,\rho_j^*<\frac{1}{k_j}$, so that $|\rho_j-\rho_j^*| <\frac{1}{k_j^2}$. Combining these two estimates, we can write
 \begin{equation}\label{xestimate}
   |\rho_j-\rho_j^*| < \min\left\{ C\lambda^{n-j}\,,\,k_j^{-2}\right\}\ .
 \end{equation}
 
 We are now ready to estimate the ratio of determinant Jacobians in \eqref{jacobratio}. Using \eqref{xyskew} and the chain rule, we see that 
 \[
  DT^n(\rho_0,\alpha_0)\;=\; \left[\begin{array}{cc}
                            \prod_{j=0}^{n-1} G'(\rho_j) & 0 \\
                            {}&{}\\
                            * & \prod_{j=0}^{n-1} T_{\rho_j}'(\alpha_j)
                           \end{array}\right]\ \ ,
 \]
and similarly for $DT^n(\rho_0^*,\alpha_0^*)$. Hence the ratio of determinant 
Jacobians at both points equals
\begin{equation}\label{twoprods}
 \frac{\det{DT^n(\rho_0,\alpha_0)}}{\det{DT^n(\rho_0^*, \alpha_0^*)}}\;=\; 
 \prod_{j=0}^{n-1} \frac{G'(\rho_j)}{G'(\rho_j^*)}\;
 \prod_{j=0}^{n-1} \frac{T_{\rho_j}'(\alpha_j)}{T_{\rho_j^*}'(\alpha_j^*)}
\end{equation}

We proceed to estimate both products in the right-hand side of \eqref{twoprods}. 

\begin{enumerate}
\item[(i)] Since $G'(\xi)=-1/\xi^2$ wherever $G$ is differentiable, each term in the first product is positive, equal to $(\rho_j^*/\rho_j)^2$, and thus we have
\[
 \left|\log{\prod_{j=0}^{n-1} \frac{G'(\rho_j)}{G'(\rho_j^*)}}\right|
 \;\leq\; 2\sum_{j=0}^{n-1} \left|\log{\rho_j}-\log{\rho_j^*}\right|
\]
The mean value inequality tells us that $\left|\log{\rho_j}-\log{\rho_j^*}\right|
\leq (k_j+1)|\rho_j-\rho_j^*|$, and therefore, by \eqref{xestimate}, we have
\begin{equation}\label{estG}
 \left|\log{\prod_{j=0}^{n-1} \frac{G'(\rho_j)}{G'(\rho_j^*)}}\right|
 \;\leq\; 4\sum_{j=0}^{n-1} k_j\min\left\{ C\lambda^{n-j}\,,\,k_j^{-2}\right\}\ .
\end{equation}

\item[(ii)] From the formulas defining the fiber maps $T_{\rho}$ (see Section \ref{SecFibermaps}), we deduce that there are only three possibilities:
\[
\frac{T_{\rho_j}'(\alpha_j)}{T_{\rho_j^*}'(\alpha_j^*)}\;=\; 
\left\{
\begin{array}{cl}
1\ , & \mbox{if $-1<\alpha_j,\alpha_j^*<0$} \\
{}&{}\\
\displaystyle{\frac{\rho_j^*\rho_{j+1}^*}{\rho_j \rho_{j+1}}}\ , & \mbox{if $0<\alpha_j<1-k_j\rho_j$ and $0<\alpha_j^*<1-k_j\rho_j^*$}\\
{}&{}\\
\displaystyle{\frac{\rho_j^*}{\rho_j}}\ , & \mbox{if $1-k_j\rho_j<\alpha_j<1$ and $1-k_j\rho_j^*<\alpha_j^*<1$}
\end{array}
\right.
\]

Whichever case occurs, we always have
\[
 \left| \log{\frac{T_{\rho_j}'(\alpha_j)}{T_{\rho_j^*}'(\alpha_j^*)}} \right| \leq 
 \left| \log{\rho_j} - \log{\rho_j^*}\right| +  
 \left| \log{\rho_{j+1}} - \log{\rho_{j+1}^*}\right| \ .
\]
This yields 
\[
 \left|\log{\prod_{j=0}^{n-1} \frac{T_{\rho_j}'(\alpha_j)}{T_{\rho_j^*}'(\alpha_j^*)} }\right|
 \leq 2\sum_{j=0}^{n} \left| \log{\rho_j} - \log{\rho_j^*}\right|\ ,
\]
Therefore, using the mean value inequality and \eqref{xestimate} just as in (i), we deduce that
\begin{equation}\label{estT}
 \left|\log{\prod_{j=0}^{n-1} \frac{T_{\rho_j}'(\alpha_j)}{T_{\rho_j^*}'(\alpha_j^*)} }\right|
 \;\leq\; 4\sum_{j=0}^{n} k_j\min\left\{ C\lambda^{n-j}\,,\,k_j^{-2}\right\}\ .
\end{equation}

\end{enumerate}
Combining the estimates \eqref{estG} and \eqref{estT}, we arrive at
\begin{equation}\label{estGT}
 \left|\log{\left(\prod_{j=0}^{n-1} \frac{G'(\rho_j)}{G'(\rho_j^*)}\;\prod_{j=0}^{n-1} \frac{T_{\rho_j}'(\alpha_j)}{T_{\rho_j^*}'(\alpha_j^*)}\right)}\right|
 \;\leq\; 8\sum_{j=0}^{n} k_j\min\left\{ C\lambda^{n-j}\,,\,k_j^{-2}\right\}\ .
\end{equation}
Applying Lemma \ref{simplelemma} with $b_j=C\lambda^j$, we deduce that the sum on the right-hand side of \eqref{estGT} is bounded by $B=\sqrt{C}/(1-\sqrt{\lambda})$. 
Thus, exponentiating both sides of this last inequality, one finally arrives 
at \eqref{jacobratio}, with $K=e^{8B}$. This completes the proof of Proposition \ref{boundjacobdist}.
\end{proof}

In what follows, we denote by $\mathrm{meas}(A)$ the Lebesgue measure of a measurable set $A\subseteq R$. 

\begin{lemma}\label{mass1} Let $A\subseteq R^{\pm}$ be a set with positive Lebesgue measure. Then there exists a constant $0<c_A<1$ such that, for every Markov $n$-tile $W$ with $T^n(W)=R^{\pm}$, we have
 \begin{equation}\label{massratio}
  \frac{\mathrm{meas}(W\cap T^{-n}(A))}{\mathrm{meas}(W)}\;\geq\; c_A\ .
 \end{equation}
\end{lemma}

\begin{proof}[Proof of Lemma \ref{mass1}] Since $T^n$ maps $W$ diffeomorphically onto $R^{\pm}$, the change-of-variables formula tells us that 
 \[
  \mathrm{meas}(A)\;=\; \iint_{W\cap T^{-n}(A)} \left|\det{DT^n}(\rho,\alpha)\right|\,d\rho d\alpha \ ,
 \]
 as well as 
 \[
 1\;=\; \mathrm{meas} (R^{\pm})\;=\; 
 \iint_{W} \left|\det{DT^n}(\rho,\alpha)\right|\,d\rho d\alpha\ .
 \]
 Applying the mean-value theorem for double integrals to both integrals above and using 
 Proposition \ref{boundjacobdist}, we deduce \eqref{massratio}, with a constant 
 $c_A$ that depends only on $\mathrm{meas}(A)$ (and the constant $K$ in \eqref{jacobratio}). 
\end{proof}

\begin{lemma}\label{mass2}
 If $B\subseteq R$ is a set with positive Lebesgue measure, then 
 \[ 
  \mathrm{meas}\left(R\setminus \bigcup_{n\geq 0} T^{-n}(B)\right)\;=\;0\ .
 \]
\end{lemma}

\begin{proof}[Proof of Lemma \ref{mass2}] Replacing $B$ by $T^{-1}(B)$ if necessary, we may assume that $B^+=B\cap R^+$ and $B^-=B\cap R^-$ both have positive measure. Let $\epsilon=\frac{1}{2}\min\{c_{B^+}\,,\,c_{B^-}\}$, where $c_{B^{\pm}}$ are the constants obtained applying Lemma \ref{mass1} to $A=B^{\pm}$. 
 
We argue by contradiction. Suppose $E=R\setminus \bigcup_{n\geq 0} T^{-n}(B)$ 
 is such that $\mathrm{meas}(E)>0$. Let $z\in E$ be a Lebesgue density point of $E$, and choose $\delta>0$ so small that the disk $D=D(z,\delta)\subset R$ satisfies 
 \begin{equation}\label{density1}
  \frac{\mathrm{meas}(D\cap E)}{\mathrm{meas}(D)}\;\geq\; 1-\epsilon\ .
 \end{equation}
 By fact (MT4) stated right after Lemma \ref{markovtiles1}, there exists a collection $\mathcal{C}$ of pairwise disjoint Markov tiles such that 
 $D=D^*\cup \bigcup_{W\in \mathcal{C}} W$, where $D^*$ has zero Lebesgue measure. For each $W\in \mathcal{C}$, there exists a positive integer $m_K$ such that $T^{m_K}(W)=R^{\pm}\supseteq B^{\pm}$. Thus, by Lemma \ref{mass1}, we have
 \[
  \mathrm{meas}\left(W\cap \bigcup_{n\geq 0} T^{-n}(B)\right) \geq 
  \mathrm{meas}\left(W\cap T^{-m_K}(B^{\pm})\right) \geq 
  c_{B^{\pm}} \mathrm{meas}(W) \geq 2\epsilon\, \mathrm{meas}(W)\ .
 \]
 Since this is true for every Markov tile in $\mathcal{C}$, we deduce that 
 $\mathrm{meas}(D\cap \bigcup_{n\geq 0} T^{-n}(B)) \geq 2\epsilon  \mathrm{meas}(D)$, that is to say,
 \begin{equation}\label{density2}
  \frac{\mathrm{meas}(D\cap (R\setminus E))}{\mathrm{meas}(D)} 
  \geq 2\epsilon\ .
 \end{equation}

But \eqref{density1} and \eqref{density2} are clearly incompatible. This contradiction shows that $\mathrm{meas}(E)=0$, and the lemma is proved.
\end{proof}

\begin{coro}\label{coroergodicity} Let $A\subset R$ be a Borel set which is strongly invariant under $T$, {\it i.e.,\/} $T^{-1}(A)=A$. If $A$ has positive Lebesgue measure, then it has full Lebesgue measure in the whole rectangle $R$.
\end{coro}

\begin{proof}[Proof of Corollary \ref{coroergodicity}] The invariance $T^{-1}(A)=A$ implies $T^{-n}(A)=A$ for all $n\geq 0$. Since $\mathrm{meas}(A)>0$, we obtain from Lemma \ref{mass2} that $\mathrm{meas}(A)=\mathrm{meas}(\bigcup_{n\geq 0}T^{-n}(A))=\mathrm{meas}(R)$.
\end{proof}

With this at hand we can finish the proof of Theorem \ref{skewprodthm}: Corollary \ref{coroergodicity} implies at once that any absolutely continuous probability measure which is invariant under $T$, is also ergodic under $T$. Therefore, the measure $\mu_T$ given by Lemma \ref{lemaacip} is ergodic. Moreover, since the support of $\mu_T$ is itself a strongly invariant subset of $R$ with positive Lebesgue measure (because it has full $\mu_T$-measure), Corollary \ref{coroergodicity} implies that it must coincide with the whole rectangle $R$ (since it is compact and it has full measure). In particular, $\mu_T$ is the unique absolutely continuous probability measure invariant under $T$, and this concludes the proof of Theorem \ref{skewprodthm}. We finish this appendix by proving Proposition \ref{genorbdensas}.

\begin{proof}[Proof of Proposition \ref{genorbdensas}] Let $B_1,B_2, \ldots, B_j,\ldots$ be a basis for the topology of $R^+\cup R^-$. For each $j\geq 1$, let $B_j^\infty = \bigcup_{n\geq 0} T^{-n}(B_j)$. Note that each $B_j^\infty\subset R^+\cup R^-$ is open, and by Lemma \ref{mass2} it has full Lebesgue measure in $R$ (in particular, it is also dense in $R$). Therefore $\mathcal{G}_0=\bigcap_{j\geq 1} B_j^\infty$ also has full Lebesgue measure in $R$. Moreover, $\mathcal{G}_0$ is a dense $G_\delta$, hence residual, subset of $R^+\cup R^-$. Finally, if $z$ is any point in $\mathcal{G}_0$, then its positive orbit $\{T^n(z):\;n\geq 0\}$ visits every basic set $B_j$, and therefore is dense in $R$. 
\end{proof}

\section{Some connections with Renormalization Theory}\label{appren}

As mentioned in the introduction, rigidity results in one-dimensional dynamics are usually related to the behaviour of some \emph{renormalization operator\/}. For circle homeomorphisms with irrational rotation number, the standard procedure is to define a renormalization operator acting on the space of \emph{commuting pairs} (see for instance \cite[Section 2]{edsonwelington1} and references therein). A fundamental principle in Renormalization Theory states that exponential convergence of renormalization orbits implies rigidity: topological conjugacies are actually smooth. For critical circle maps with a single critical point, this principle has been established by the first author and de Melo for Lebesgue almost every irrational rotation number \cite[First Main Theorem]{edsonwelington1}, and extended later by Khanin and Teplinsky to cover all irrational rotation numbers \cite[Theorem 2]{khaninteplinsky}. Adapting these previous approaches, this fundamental principle has been recently established for multicritical circle maps in \cite[Theorem A]{EG}.

Given a bi-critical circle map $f$ with irrational rotation number $\rho$, unique invariant Borel probability measure $\mu$ and critical points $c_1$ and $c_2$, let $\alpha\in (0,1)$ be such that the two connected components of $S^1\setminus\{c_1,c_2\}$ have $\mu$-measures equal to $\alpha$ and $1-\alpha$ respectively. We say that the pair $(\rho,\alpha)$ is the \emph{signature} of $f$. It is not difficult to see that the skew product $T$, constructed in \S \ref{Secskew} of the present paper, coincides with the action of the renormalization operator on the signature $(\rho,\alpha)$. The expanding behaviour of the fiber maps $T_{\rho}$ from \S \ref{SecFibermaps} suggests both coexistence of periodic orbits and chaotic behaviour inside topological classes of bi-critical commuting pairs (since, by Yoccoz's result \cite{yoccoz}, the topological classes are obtained just by fixing the rotation number $\rho$). In the recent preprint \cite{yampolsky5}, Yampolsky was able to prove that if $(\rho,\alpha)$ is any given periodic orbit under $T$, say of period $p\in\nt$, then there exists a real-analytic bi-critical circle map, whose signature equals $(\rho,\alpha)$, which is a periodic orbit (of the same period $p$) for the renormalization operator \cite[Theorem 2.8]{yampolsky5}. These periodic orbits are \emph{hyperbolic}, with local stable manifolds of codimension $2$\, \cite[Theorem 8.3]{yampolsky5}. Moreover, each local stable manifold is obtained precisely by fixing the signature $(\rho,\alpha)$ (again, see \cite[Theorem 8.3]{yampolsky5}), which is compatible with the expanding behaviour of the skew product $T$, as discussed in Appendix \ref{apperg}.

\medskip

\section*{Acknowledgements}

This work had its seeds planted in the summer of 2016, when both authors visited Imperial College London. We wish to express our thanks to that institution for its hospitality, and especially to Sebastian van Strien for his kind invitation and support.

\end{document}